\documentclass{amsart}
\usepackage{amssymb, color}
\usepackage[shortlabels]{enumitem}
\usepackage[ascii]{inputenc}
\usepackage[all]{xy}
\usepackage[dvipsnames]{xcolor} 
\usepackage{graphicx}
\usepackage[hidelinks]{hyperref}
\usepackage{bm}
\title
[Controlling cardinal characteristics without adding reals]
{Controlling cardinal characteristics\\ without adding reals}
\date{2020-10-20}
\author{Martin Goldstern}
\address{TU Wien, Institute of
Discrete Mathematics and Geometry, Wiedner Hauptstrasse 8-10/104, 1040 Wien, Austria.}
\email{goldstern@tuwien.ac.at}
\urladdr{http://www.tuwien.ac.at/goldstern/}

\author{Jakob Kellner}
\address{TU Wien, Institute of
Discrete Mathematics and Geometry, Wiedner Hauptstrasse 8-10/104, 1040 Wien, Austria.}
\email{jakob.kellner@tuwien.ac.at}
\urladdr{http://dmg.tuwien.ac.at/kellner/}

\author{Diego A. Mej\'{i}a}
\address{Creative Science Course (Mathematics), Faculty of Science, Shizuoka University, Ohya 836, Suruga-ku, Shizuoka-shi, Japan 422-8529.}
\email{diego.mejia@shizuoka.ac.jp}
\urladdr{http://www.researchgate.com/profile/Diego\_Mejia2}

\author{Saharon Shelah}
\address{The Hebrew University of Jerusalem, Einstein Institute of Mathematics, Edmond J. Safra Campus, Givat Ram,  Jerusalem, 91904, Israel, and Rutgers University, Department of Mathematics, New Brunswick, NJ 08854, USA.}
\email{shlhetal@math.huji.ac.il}
\urladdr{http://shelah.logic.at}

\thanks{This work was supported by the following grants:
Austrian Science Fund (FWF): project number I3081, P29575
(first author) and P30666
(second author);
Grant-in-Aid for Early Career Scientists 18K13448, Japan Society for the Promotion of Science (third author); Israel Science Foundation (ISF) grant no: 1838/19 and NSF grant no: DMS 1833363 (fourth author). This is publication number 1166 of the fourth author.}



\newcommand{\LCU}{\textnormal{LCU}}
\newcommand{\COB}{\textnormal{COB}}
\DeclareMathOperator{\add}{add}
\DeclareMathOperator{\cov}{cov}
\DeclareMathOperator{\non}{non}
\DeclareMathOperator{\cof}{cof}
\DeclareMathOperator{\cf}{cof}
\DeclareMathOperator{\crit}{cr}

\DeclareMathOperator{\dom}{dom}
\newcommand{\Null}{\mathcal N}
\newcommand{\Meager}{\mathcal M}
\newcommand{\addN}{{\ensuremath{\add(\Null)}}}
\newcommand{\cofN}{{\ensuremath{\cof(\Null)}}}
\newcommand{\covN}{{\ensuremath{\cov(\Null)}}}
\newcommand{\nonN}{{\ensuremath{\non(\Null)}}}
\newcommand{\addM}{{\ensuremath{\add(\Meager)}}}
\newcommand{\cofM}{{\ensuremath{\cof(\Meager)}}}
\newcommand{\covM}{{\ensuremath{\cov(\Meager)}}}
\newcommand{\nonM}{{\ensuremath{\non(\Meager)}}}

\newcommand{\cfrak}{\mathfrak{c}}
\newcommand{\bfrak}{\mathfrak{b}}
\newcommand{\dfrak}{\mathfrak{d}}
\newcommand{\gfrak}{\mathfrak{g}}
\newcommand{\hfrak}{\mathfrak{h}}

\newcommand{\mfrak}{\mathfrak{m}}
\newcommand{\pfrak}{\mathfrak{p}}
\newcommand{\tfrak}{\mathfrak{t}}
\newcommand{\xfrak}{\mathfrak{x}}
\newcommand{\precal}{\textnormal{cal}}

\newcommand{\plike}{$\mathfrak{t}$-like}
\newcommand{\tlike}{\plike}
\newcommand{\mlike}{$\mathfrak{m}$-like}
\newcommand{\hlike}{$\mathfrak{h}$-like}

\newcommand{\la}{\langle}
\newcommand{\ra}{\rangle}
\newcommand{\Qhor}{Q^2}

\theoremstyle{plain}
  \newtheorem{theorem}[equation]{Theorem}
  \newtheorem{corollary}[equation]{Corollary}
  \newtheorem{lemma}[equation]{Lemma}
   
   \newtheorem{fact}[equation]{Fact}

  \newtheorem{clm}[equation]{Claim}

\theoremstyle{definition}
  \newtheorem{definition}[equation]{Definition}
  
  \newtheorem*{example*}{Example}
  \newtheorem{remark}[equation]{Remark}
  \newtheorem*{remark*}{Remark}
  
  \newtheorem*{remarks*}{Remarks}
  
  \newtheorem*{notation*}{Notation}
  
  \newtheorem{assumption}[equation]{Assumption}

\numberwithin{equation}{section}

\newcommand{\innitialmark}[1]{{#1}^\text{pre}}
\newcommand{\finalmark}[1]{{#1}^\text{fin}}
\newcommand{\Ppre}{\innitialmark{P}}
\newcommand{\Pfin}{\finalmark{P}}

\begin{document}
\begin{abstract}
    We investigate the behavior of cardinal characteristics
    of the reals under extensions that do not add new ${<}\kappa$-sequences (for some regular $\kappa$).

    As an application, we show that consistently
    the following cardinal characteristics can be
    different: The (``independent'') characteristics in Cicho\'n's diagram, plus $\aleph_1<\mathfrak m<\mathfrak p<\mathfrak h<\addN$.
    (So we get thirteen different values, including $\aleph_1$ and continuum).

    We also give
    constructions to alternatively separate other MA-numbers (instead of $\mfrak$), namely:
    MA for
    $k$-Knaster from MA for $k+1$-Knaster;
    and MA for the union of all $k$-Knaster forcings from MA for precaliber.
\end{abstract}
\maketitle


\section{Introduction}

In this work we investigate how to preserve and how to change certain
cardinal characteristics of the continuum in NNR extensions, i.e., extensions that do not
add reals; or more generally that do not add ${<}\kappa$-sequences of ordinals
for some regular $\kappa$. It is known that the
``Blass-uniform'' characteristics (see Definition~\ref{def:blassu})
tend to keep their values in such extensions
(cf.\ Mildenberger's~\cite[Prop.~2.1]{MR1625907}),
and we give some explicit results in that
direction.
Other cardinal characteristics tend to keep a value $\theta$ only if
$\theta<\kappa$. We will use this effect to combine various forcing notions
(most of them already known) to get models with many simultaneously different
``classical'' characteristics.

In particular, we look at the entries of Cicho\'n's diagram, which we call \emph{Cicho\'n-characteristics} (see Figure~\ref{fig:cichon}, we assume that the reader is familiar with this diagram), and the following characteristics:
\begin{definition}\label{DefChar1}
    Let $\mathcal{P}$ be a class of posets.
    \begin{enumerate}[(1)]
        \item $\mfrak(\mathcal{P})$ denotes the minimal cardinal where Martin's axiom for the posets in $\mathcal{P}$ fails. More explicitly, it is the minimal $\kappa$ such that, for some poset $Q\in\mathcal{P}$, there is a collection $\mathcal{D}$ of size $\kappa$ of dense subsets of $Q$ such that there is no filter in $Q$ intersecting all the members of $\mathcal{D}$.
        \item $\mfrak:=\mfrak(\textnormal{ccc})$.
        \item Write $a\subseteq^* b$ iff $a\smallsetminus b$ is finite. Say that $a\in[\omega]^{\aleph_0}$ is a \emph{pseudo-intersection} of $F\subseteq[\omega]^{\omega}$ if $a\subseteq^* b$ for all $b\in F$.
        \item The \emph{pseudo-intersection number $\pfrak$} is the smallest size of a filter base of a free filter on $\omega$ that has no pseudo-intersection in $[\omega]^{\aleph_0}$.
        \item The \emph{tower number} $\tfrak$ is the smallest order type of a $\subseteq^*$-decreasing sequence in $[\omega]^{\aleph_0}$ without pseudo-intersection.
        \item The \emph{distributivity number} $\hfrak$ is the smallest size of a collection of dense subsets of $([\omega]^{\aleph_0},\subseteq^*)$ whose intersection is empty.
        \item A family $D\subseteq[\omega]^{\aleph_0}$ is \emph{groupwise dense} if
        \begin{enumerate}[(i)]
            \item $a\subseteq^* b$ and $b\in D$ implies $a\in D$, and
            \item whenever $(I_n:n<\omega)$ is an interval partition of $\omega$, there is some $a\in[\omega]^{\aleph_0}$ such that $\bigcup_{n\in a}I_n\in D$.
        \end{enumerate}
        The \emph{groupwise density number $\gfrak$} is the smallest size of a collection of groupwise dense sets whose intersection is empty.
    \end{enumerate}
\end{definition}

\newcommand{\mye}{*+[F.]{\phantom{\lambda}}}
\begin{figure}
  \centering
\[
\xymatrix@=4.5ex{
&            \covN\ar[r] & \nonM \ar[r]      &  \mye \ar[r]     & \cofN\ar[r] &2^{\aleph_0} \\
&                               & \mathfrak b\ar[r]\ar[u]  &  \mathfrak d\ar[u] &              \\
  \aleph_1\ar[r] & \addN\ar[r]\ar[uu] & \mye\ar[r]\ar[u] &  \covM\ar[r]\ar[u]& \nonN\ar[uu]
}
\]
    \caption{\label{fig:cichon}Cicho\'n's diagram with the
two ``dependent'' values removed, which are
$\addM=\min(\mathfrak b, \covM)$
and $\cofM=\max(\nonM,\mathfrak d)$.
An arrow $\mathfrak x\rightarrow \mathfrak y$ means that
ZFC proves $\mathfrak x\le \mathfrak y$.}
\end{figure}
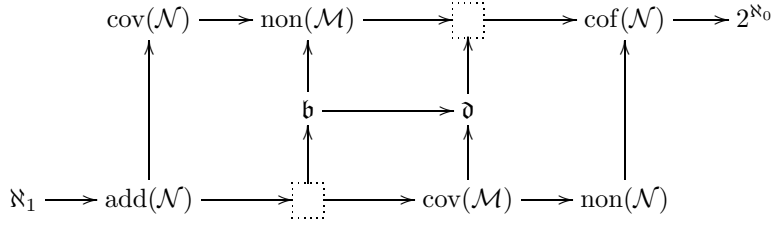

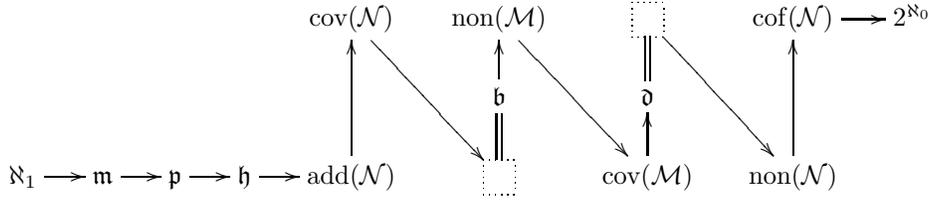
\begin{figure}
  \centering
\[
\xymatrix@=3.5ex{
&&&&            \covN\ar[rdd] & \nonM \ar[rdd]      &  \mye\ar@{=}[d]\ar[ddr]      & \cofN\ar[r] &2^{\aleph_0} \\
&&&&                               & \mathfrak b\ar[u]  &  \mathfrak d &              \\
\aleph_1\ar[r] & \mfrak\ar[r] & \pfrak\ar[r]&\hfrak\ar[r]
& \addN\ar[uu] & \mye\ar@{=}[u] &  \covM\ar[u]& \nonN\ar[uu]
}
\]
    \caption{\label{fig:result}The model we construct in this paper; here $\mathfrak x\rightarrow \mathfrak y$ means
 that $\mathfrak x<\mathfrak y$.
    Any number of the $<$ signs can be replaced by $=$ as desired.}
\end{figure}

We are aware of the following ZFC provable relations between these cardinals:
\begin{equation}
    \label{eq:uqw523}
\mfrak\leq\pfrak=\tfrak\leq\hfrak\leq\gfrak,\quad \mfrak\leq\addN,\quad
\tfrak\leq\addM,\quad
\hfrak\leq\bfrak,\quad
\gfrak\leq\cof(\dfrak).
\end{equation}
Also, with the exception of $\mfrak$ and $\dfrak$, all the cardinals in~\eqref{eq:uqw523} are known to be regular (and uncountable), $2^{<\tfrak}=\cfrak$ and $\gfrak\le\cof(\cfrak)$. For details see e.g.\ Blass~\cite{Blass}; but for $\pfrak=\tfrak$ see~\cite{MSpt} with Malliaris,\footnote{However, only the trivial inequality $\pfrak\leq\tfrak$ is used in this text.} and
$\gfrak\leq\cof(\dfrak)$ follows from the fact that
$\cof((\omega,<)^\omega/\mathcal U)=\cof(\dfrak)$ for some ultrafilter $\mathcal U$, due to Canjar~\cite{MR1036675}, and $\gfrak\le \cof((\omega,<)^\omega/\mathcal U)$ for any ultrafilter $\mathcal U$, due to Blass and Mildenberger~\cite{MR1777781}.

Recently~\cite{GKS} constructed, assuming four strongly compact cardinals, a ZFC model where
the ten (non-dependent) Cicho\'n-characteristics are pairwise different.
This orders the characteristics as shown in Figure~\ref{fig:result}. 
In~\cite{GKMS2} we give a construction that does not require
large cardinals.

%

To continue with this line of work, we ask whether other classical cardinal characteristics of the continuum can be included and forced to be pairwise different. Our main result is that
we can additionally force that $\aleph_1 < \mfrak<\pfrak<\hfrak=\gfrak<\addN$, thus yielding a model where 13 classical cardinal characteristics are pairwise different.

We now give an outline of this paper:
\smallskip

\noindent\textbf{S.~\ref{sec:prelim}, p.~\pageref{sec:prelim}: Preliminaries.} We review some aspects of the Cicho\'n's Maximum construction (the construction from~\cite{GKMS2} that gives 10 different values in Cicho\'n's diagram). In particular, we mention Blass-uniform characteristics and the LCU and COB properties.
\smallskip

\noindent\textbf{S.~\ref{sec:nnr}, p.~\pageref{sec:nnr}: NNR extensions.} We define some classes of cardinal characteristics and show how they are affected (or unaffected) by extensions that do not add new ${<}\kappa$-sequences
for some regular $\kappa$; in particular: under ${<}\kappa$-distributive forcing extensions; and when intersecting the poset with some ${<}\kappa$-closed elementary submodel.\smallskip

\noindent\textbf{S.~\ref{sec:ma}, p.~\pageref{sec:ma}: $\bm{\mfrak}$.}
Using classical methods
of Barnett and Todor\v{c}evi\'{c}~\cite{TodorCell,To93,Barnett},
we modify the Cicho\'n's Maximum construction
to additionally force $\mfrak=\lambda_{\mfrak}$ for any given regular value $\lambda_{\mfrak}$ between $\aleph_1$ and $\addN$.


In addition to $\mfrak$,
we can control the Knaster-numbers $\mfrak(k\textnormal{-Knaster})$ as well.
But this
does not give a larger number of simultaneously
different characteristics (as all Knaster numbers bigger than $\aleph_1$ have the same value,
which is also the value of $\mfrak(\textnormal{precaliber})$).
We give models for all possible constellations (at least for regular $\lambda$): All Knaster numbers (and $\mfrak(\textnormal{precaliber})$) can be
$\aleph_1$. And
there can be a $k\ge 1$ such that
$\mfrak(\ell\textnormal{-Knaster})=\aleph_1$ for all
$1\le \ell<k$ and $\mfrak(\ell\textnormal{-Knaster})=\lambda$
for $\ell\ge k$. (For notational convenience, we identify 1-Knaster with ccc.)
%
%
\smallskip

\noindent\textbf{S.~\ref{sec:precaliber}, p.~\pageref{sec:precaliber}: $\bm{\mfrak(\textbf{precaliber})}$.} We deal with a case that was
left open in the previous section:
We construct a model where all Knaster numbers
are $\aleph_1$, and the precaliber number is some regular $\lambda>\aleph_1$.
\smallskip

\noindent\textbf{S.~\ref{sec:h}, p.~\pageref{sec:h}: $\bm{\hfrak}$.} Given a poset $P$, we show how to obtain a complete subposet $P'$ of $P$ forcing smaller values to $\gfrak$ (and thus $\hfrak\le\gfrak$), 
while preserving certain other values for cardinal characteristics already forced by $P$. This method allows us to get $\pfrak=\hfrak=\gfrak$.
\smallskip

\noindent\textbf{S.~\ref{sec:p}, p.~\pageref{sec:p}: $\bm{\pfrak}$.} Based on a result with Dow~\cite{longlow}, we show that the product of a $\xi$-cc poset $P$ with
the poset $\xi^{<\xi}$ may add a tower of length $\xi$, while preserving the cardinal $\hfrak$ above $\xi$ and the values for the Cicho\'n-characteristics that were already forced by $P$.

This allows us to prove the main theorem, thirteen pairwise different characteristics.

\noindent\textbf{S.~\ref{sec:ext}, p.~\pageref{sec:ext}: Extensions.} We remark on
alternative initial  forcings (i.e., forcings for the left hand side of Cicho\'n's diagram) and an alternative order.
\smallskip




\begin{notation*}
When we are investigating a
characteristic
$\mathfrak x$ and plan to force a specific value
to it, we will usually call this value $\lambda_{\mathfrak x}$.
Let us stress that
calling a cardinal $\lambda_{\mathfrak x}$ is \emph{not} an implicit \emph{assumption}
that
$P\Vdash \mathfrak x = \lambda_{\mathfrak x}$ for the $P$
under investigation; it is just an (implicit) declaration
of intent.
\end{notation*}

\subsection*{Acknowledgements}
We would like to thank Teruyuki Yorioka for pointing out the reference~\cite{Barnett}, which is cited in Section~\ref{sec:ma}.

We would like to thank an anonymous referee for suggesting several improvements (and for pointing out that $\gfrak\leq\cof(\dfrak)$).

\section{Preliminaries}\label{sec:prelim}

We mention some of the required definitions and
constructions from~\cite{GKS} and~\cite{GKMS2}.
We will not give all required proofs and not even the
complete construction, as it is rather involved.
We will have to assume that the reader either
knows this construction, or is willing to accept it as
a blackbox.

\subsection{LCU and COB, the initial forcing \texorpdfstring{$\Ppre$}{ } for the left side}\label{subsec:blass}


\begin{definition}\label{def:blassu}
   A \emph{Blass-uniform cardinal characteristic} is a characteristic of the form
   \[\mathfrak d_R:=\min\{|D|:D\subseteq\omega^\omega\text{\ and }(\forall x\in\omega^\omega)\,(\exists y\in D) \ xRy\}\]
   for some Borel\footnote{We could just as well assume that $R$
is analytic or co-analytic. More specifically, for
all results in this paper, it is enough
to assume that $R$ is absolute between the extensions we consider;
in our case between extensions that do not add new reals. So even
projective relations would be OK.
However, all concrete relations that we will actually use are Borel, even of very low rank. Regarding ``on $\omega^\omega$'', see Remark~\ref{rem:omegaomega}.} $R$. To avoid trivialities, we will only consider relations $R$ for which $\dfrak_R$ (and the dual $\bfrak_R$ below) are well defined.\footnote{I.e., $(\forall x\in\omega^\omega)\,(\exists y,z\in\omega^\omega)\, x R y\,\wedge\,\lnot z  R x$.}
\end{definition}

Such characteristics have been studied systematically since at least the 1980s by many authors, including
Fremlin~\cite{zbMATH03891346}, Blass~\cite{MR1234278, Blass} and
Vojt\'{a}\v{s}~\cite{MR1234291}.

Note that its dual cardinal
\[{\mathfrak b}_R:=\min\{|F|:F\subseteq\omega^\omega\text{\ and }(\forall y\in\omega^\omega)\,(\exists x\in F) \ \neg xRy\}\]
is also Blass-uniform because $\mathfrak{b}_R=\mathfrak{d}_{R^\perp}$ where  $xR^\perp y$ iff $\neg(yRx)$.

\begin{remark*}
All Blass-uniform characteristics in this paper, and many others,
such as those in Blass' survey~\cite{Blass} or those in~\cite{GoSh:448}, are in fact
of the form $\bfrak_R$ or $\dfrak_R$ for some $\Sigma^0_2$ relation~$R$
which is invariant under finite modifications of its arguments.
When we restrict to such relations,  there is no ambiguity
as to which
Blass-uniform cardinal characteristics are of the form $\bfrak_R$ and which
are of the form $\dfrak_R$.   It was shown by Blass~\cite{MR1234278}
that for such relations $R$ we must have $\bfrak_R\le \nonM$
and $\dfrak_R\ge \covM$, thus $\bfrak_R$ is always on the left side
of Cicho\'n's diagram, and $\dfrak_R$ is on the right side.
\end{remark*}

\begin{remark}\label{rem:omegaomega}
   It can be more practical to
   consider more generally
   relations on $X\times Y$
   for some Polish spaces $X$, $Y$ other than $\omega^\omega$, in particular as
   many examples of Blass-uniform cardinals are
   naturally defined in such spaces.

   To cover such cases,
   one can either modify the definition, or use a
   Borel isomorphisms to translate the relation to $\omega^\omega$.
\end{remark}


The Cicho\'n-characteristics are all Blass-uniform,
defined by
natural\footnote{The relations $R$ used to define
    the following characteristics are ``natural'', but not entirely ``canonical''. For example, a different choice
    of a natural relation $R$ such that $\mathfrak b_R=\mathfrak s$ leads to a different dual $\mathfrak d_R=\mathfrak r_\sigma$. See ~\cite[Example 4.6]{Blass}.}
relations. Accordingly, they come in pairs
$(\mathfrak b_R,\mathfrak d_R)$ for the according
Borel relation $R$:\\
    $(\addN,\cofN)$, $(\covN,\nonN)$,
    $(\addM,\cofM)$, $(\nonM,\covM)$,
     and $(\mathfrak b,\mathfrak d)$.
        (The last pair, for example, is defined by
        eventual domination $\le^*$.)

Another example for a Blass-uniform pair is
$(\mathfrak s, \mathfrak r)=(\mathfrak b_R,\mathfrak d_R)$ where
$\mathfrak s$ is splitting number and $\mathfrak r$ the reaping number
and  $R$ is the relation on $[\omega]^{\aleph_0}$ that states $xRy$ iff ``$x$ does not split $y$''.

We will often have a situation where
$(\mathfrak b_R,\mathfrak d_R)=(\lambda,\mu)$ is ``strongly
witnessed'', as follows:

\begin{definition}\label{def:lcucob}
Fix a Borel relation $R$,
$\lambda$ a regular cardinal and
$\mu$ an arbitrary cardinal. We define two properties:\footnote{In \cite{diegoetal} (and in other related work), a family with $\LCU_R(\lambda)$ is said to be \emph{strongly $\lambda$-$R$-unbounded of size $\lambda$}, while a family with $\COB_R(\lambda,\mu)$ is said to be \emph{strongly $\lambda$-$R$-dominating of size $\mu$}.}

\begin{description}[labelindent=0pt] 
\item[Linearly cofinally unbounded]
    $\LCU_R(\lambda)$ means: There is a family
    $\bar f=(f_{\alpha}:\alpha<\lambda)$  of
    reals
    such that:
    \begin{equation}\label{eq:LCU}
    (\forall g\in \omega^\omega)\, (\exists\alpha\in\lambda)\,(\forall \beta\in \lambda\setminus \alpha) \ \lnot f_{\beta} R g.
    \end{equation}

\item[Cone of bounds]
    For $\lambda\le\mu$,
    $\COB_R(\lambda,\mu)$ means:\footnote{Note that $\COB_R(\lambda,\mu)$
    for $\lambda>\mu$ would violate our assumption
    that $\bfrak_R$ is well-defined: In that case,
    $\COB_R(\lambda,\mu)$ would imply that $\mu$ has a top element with respect to the order $\trianglelefteq$, so there is an $x\in\omega^\omega$ with $y R x$ for all $y$.}
    There is a $\mathord<\lambda$-directed partial order $\trianglelefteq$ on $\mu$,\footnote{I.e., every subset of $\mu$ of
    cardinality ${<}\lambda$ has a $\trianglelefteq$-upper bound}
    and a family $\bar g= (g_{s}:s\in \mu)$ of
    reals such that
    \begin{equation}\label{eq:COB}
    (\forall f\in\omega^\omega)\, (\exists s\in \mu)\, (\forall t\trianglerighteq s)\  f R g_{t}.
    \end{equation}

\end{description}
\end{definition}

\begin{fact}\label{fact:bla23424}
    $\LCU_R(\lambda)$ implies $\mathfrak b_R \le \lambda\leq\mathfrak d_R$.

    $\COB_R(\lambda,\mu)$
    implies $\mathfrak b_R \ge \lambda$ and
    $\mathfrak d_R \le \mu$.
\end{fact}

\begin{remark}\label{rem:cobmonoton}
$\COB_R(\lambda,\mu)$ clearly implies $\COB_R(\lambda',\mu)$ whenever $\lambda'\le \lambda$.
The property $\COB_R(2,\mu)$, the weakest of these notions, just says that
there is a witness for $\mathfrak d_R\le \mu$,  or in other words: there is
an $R$-dominating\footnote{Formally:
$D\subseteq \omega^\omega$ is $R$-dominating iff
$(\forall x\in \omega^\omega)\,(\exists y\in D) \ x R y$.}
family of size $\mu$.

Also, $\COB_R(\lambda,\mu)$ implies $\COB_R(\lambda,\mu')$ whenever $\mu'\ge \mu$.
\end{remark}

Informally, we call the objects
$\bar f$
in the definition
of $\LCU$
and $(\trianglelefteq,\bar g)$ for $\COB$  ``strong witnesses'',
and say that the corresponding cardinal inequalities (or equalities)
are ``strongly witnessed''.

In~\cite{GKS} (building on~\cite{GMS}) the following is shown:
\begin{lemma}\label{lem:oldleft}
   Assume GCH and $\aleph_1<\nu_1<\nu_2<\nu_3<\nu_4<\theta_\infty$
   are all successors of regular cardinals.
   Then there is a ccc countable support iteration $\Ppre$
   of length $\theta_\infty+\theta_\infty$
   forcing that
   \[\aleph_1<\addN=\nu_1<\covN=\nu_2<\bfrak=\nu_3<\nonM=\nu_4<\cfrak=\theta_\infty.\]
   Moreover, all the equalities are strongly witnessed;
   all iterands in $P$ are $(\sigma, k)$-linked (see Definition~\ref{DefKnaster}) for all $k$;
   and in the first $\theta_\infty$ many steps we add Cohen
   reals.
\end{lemma}

In this work, we will modify this construction $\Ppre$ to
get similar iterations $P$ that allow us to add
additional characteristics. We claim that
these modifications will not change the fact that
the characteristics in Lemma~\ref{lem:oldleft} are strongly witnessed.
A reader who doesn't know the proof of Lemma~\ref{lem:oldleft} will hopefully trust us on this;
for the others we give
the (simple) argument:
\begin{itemize}
    \item We get the required $\COB$ properties simply
       by bookkeeping, when forcing with ``partial random'',
       or ``partial eventually different'', etc., forcings.
       This will not change when we add
       additional iterands (as long as,
       cofinally often, we choose the iterands as
       in the original construction).
    \item Fix a (left hand) Cicho\'n-characteristic $\xfrak$
        other than $\bfrak$.
        We get the strong witness $\LCU_R(\nu)$
        (for $R$ a relation connected to $\xfrak$ and
        $\nu$ the according $\nu_i$)
        because all the iterands are ``$(\nu,R)$-good''.

        Any forcing of size ${<}\nu$ is automatically
        good, so adding small iterands will not be a problem.

        Also, $\sigma$-centered forcings are always good
        for the characteristics $\addN$ and $\covN$.
    \item For $\bfrak$, it is more cumbersome to prove
        $\LCU_R(\nu_3)$,
        but at least it is clear that adding additional
        iterands of size ${<}\nu_3$ will not interfere with
        the proof.
\end{itemize}
So we can summarize:
\begin{clm}\label{claim:oldtonew}
We can add to $\Ppre$
arbitrary iterands that all are
\begin{itemize}
\item either of size ${<}\nu_1$,
\item or $\sigma$-centered and of size ${<}\nu_3$,
\end{itemize}
        and still force strong witnesses for the
        Cicho\'n-characteristics of Lemma~\ref{lem:oldleft}.
\end{clm}
(Of course these new iterands have to be added in a way
so that we still use the old iterands unboundedly
often; we cannot just add new iterands at the end.)

\begin{remark*}
  Instead of the construction of~\cite{GKS},
  one can use alternative constructions that
  require weaker assumptions, cf.\ Section~\ref{sec:alternatives}.
\end{remark*}

\subsection{The Cicho\'n's Maximum construction}

As before, we will not require or describe the construction in detail, but only present the basic structure and certain properties.

The following is the main Theorem (3.1) of~\cite{GKMS2}. As we
will use the assumptions of the theorem repeatedly, we make them explicit:
\begin{assumption}\label{asm:main}
Assume GCH, and that
\begin{gather*}
\aleph_1\le\kappa
\le \lambda_\addN
\le \lambda_\covN
\le \lambda_\bfrak
\le \lambda_\nonM\le\\
\le \lambda_\covM
\le \lambda_\dfrak
\le \lambda_\nonN
\le \lambda_\cofN
\le \lambda_\infty
\end{gather*}
are regular cardinals, with the
possible exception of $\lambda_\infty$, for which we only require
$\lambda_\infty^{<\kappa}=\lambda_\infty$.
\end{assumption}

\begin{theorem}\label{thm:old}
Under these assumptions, there is a ccc poset $\Pfin$ forcing strong witnesses for
\begin{gather*}
\aleph_1\leq\addN=\lambda_\addN\leq\covN=\lambda_\covN\leq\bfrak=\lambda_\bfrak\leq\nonM=\lambda_\nonM\leq \\
\covM=\lambda_\covM\leq\dfrak=\lambda_\dfrak\leq\nonN=\lambda_\nonN\leq\cofN=\lambda_\cofN\leq\cfrak=\lambda_\infty.
\end{gather*}
\end{theorem}
Note that $\kappa$ does not make much sense in this theorem,
as you can just set $\kappa=\aleph_1$ (resulting in the weakest requirement
$\lambda_\infty^{\aleph_0}=\lambda_\infty$).
Indeed this is what is done in~\cite{GKMS2} (where $\kappa$ is not mentioned at all);
but mentioning $\kappa$ explicitly here will be useful in Lemma~\ref{lem:oldnew} below.

The construction in~\cite{GKMS2} is as follows:
\begin{enumerate}[(A)]
    \item Pick a sequence of successors of regular cardinals (strictly) above $\lambda_\infty$:
\[
\xi_1<\nu_1<\xi_2<\nu_2<\xi_3<\nu_3<\xi_4<\nu_4<\theta_\infty,
\]
    \item Start with any initial $\kappa$-cc poset $\Ppre$
    for the ``left hand side'', which forces ``strong witnesses'' for
\[
\addN=\nu_1<\covN=\nu_2<\bfrak=\nu_3<\nonM=\nu_4<\cfrak=\theta_\infty.
\]
(So we can use the forcing of Lemma~\ref{lem:oldleft},
or any modification satisfying Claim~\ref{claim:oldtonew}.)
\end{enumerate}

The proof in~\cite{GKMS2} can then be formulated as the following:
\begin{lemma}\label{lem:oldnew}
  Under Assumption~\ref{asm:main},
  and given a forcing $\Ppre$ as in (A) and (B),
  there is a
  ${<}\kappa$-closed\/\footnote{%
    \cite{GKMS2} uses the case
    $\kappa=\aleph_1$, so we get only a countably closed
    $N^*$. But the the proof there works for any uncountable
    regular $\kappa$, with only the trivial change: We let $N_8$ be a ${<}\kappa$-closed model of size $\lambda_\infty$, and note that then $N^*$ is ${<}\kappa$-closed as well.}
elementary submodel $N^*$ of $H(\chi)$ such that  $\Pfin:=\Ppre\cap N^*$ witnesses Theorem~\ref{thm:old}.
\end{lemma}
(As usual, $\chi$ is a sufficiently large, regular cardinal.)

\subsection{History}\label{subsec:history} We briefly remark on the history of the result of this section.

A (by now) classical series of
results by various authors~\cite{MR719666,MR1233917,MR781072,JS,MR1022984,Krawczyk83,MR613787,MR735576,MR697963}
(summarized by Bartoszy{\'{n}}ski and Judah~\cite{BaJu}) shows that
any assignments of $\{\aleph_1,\aleph_2\}$ to the Cicho\'n-characteristics
that satisfy the well known ZFC restrictions is consistent.
This leaves the questions how to show that many values can be simultaneously different.
The ``left hand side'' part was done in~\cite{GMS} and uses
eventually different forcing $\mathbb E$
to ensure  $\nonM\ge \lambda_\nonM$
and ultrafilter-limits of $\mathbb E$  to show that
$\mathfrak{b}$ remains small.
It relies heavily on the notion of goodness, introduced in~\cite{JS} (with Judah) and by Brendle~\cite{Br}, and summarized in e.g.~\cite{GMS} or~\cite{CM19} (with Cardona).

Based on this construction,
\cite{GKS} uses Boolean ultrapowers to
get simultaneously different values for all (independent) Cicho\'n-characteristics, modulo four strongly compact cardinals.

For this, the construction for the left hand side first has to be modified to get a ccc forcing starting with a ground model satisfying GCH.

Then Boolean ultrapowers are applied
to separate the cardinals on the right side.
\cite{KTT} (with T\v{a}nasie and Tonti)
gives an introduction to the Boolean ultrapower construction.
Such Boolean ultrapowers are applied four times, once for each pair of cardinals on the right side that are separated.

For this it is required that there is a strongly compact cardinal between two values corresponding to adjacent cardinals characteristics on the left side, so the cardinals on this side are necessarily very far apart.
\cite{diegoetal} improves the left hand side construction of \cite{GMS} to include
$\covM<\mathfrak d=\nonN=\cfrak$. 
This is achieved by using matrix iterations of partial Frechet-linked
posets (the latter concept is originally from~\cite{mejiavert}).
Then the same method of Boolean ultrapowers as before  can be applied, in the same way, to force different values for all Cicho\'n-characteristics, modulo three strongly compact cardinals.

Finally, in~\cite{GKMS2}
we can get the result  without assuming large cardinals;
this is the construction we use in this paper.

\section{Cardinal characteristics in extensions without new \texorpdfstring{${<}\kappa$}{<kappa}-sequences}\label{sec:nnr}

Let us consider ${<}\kappa$-distributive forcing extensions for some regular $\kappa$.
(In particular these extensions are NNR, i.e., do not
add new reals.)
For such extensions, we can also
preserve strong witnesses in some cases:

\begin{lemma}\label{lem:blassdistr}
Assume that $Q$ is $\theta$-cc and ${<}\kappa$-distributive
for $\kappa$ regular uncountable, and let $\lambda$ be a regular cardinal
and $R$ a Borel relation.
\begin{enumerate}
    \item If $\LCU_R(\lambda)$, then
    $Q\Vdash\LCU_R(\cof(\lambda))$.

    So if additionally $\lambda\le\kappa$ or $\theta\le\lambda$,
    then $Q\Vdash\LCU_R(\lambda)$.
    \item
    If $\COB_R(\lambda,\mu)$ and
    either $\lambda\le\kappa$ or $\theta\le\lambda$,
    then $Q\Vdash \COB_R(\lambda,|\mu|)$.

    So for any $\lambda$, $\COB_R(\lambda,\mu)$ implies
    $Q\Vdash \COB_R(\min(|\lambda|,\kappa),|\mu|)$.
\end{enumerate}
\end{lemma}

\begin{proof}
For (1) it is enough to assume that $Q$ does not add
reals:
Take a strong witness for $\LCU_R(\lambda)$. This object
still satisfies~\eqref{eq:LCU} in the $Q$-extension
(as there are no new reals),
but the index set will generally not be regular any more;
we can just take a cofinal subset of order type $\cof(\lambda)$
which will still satisfy~\eqref{eq:LCU}.

Similarly, a strong witness for $\COB_R(\lambda,\mu)$
still satisfies~\eqref{eq:COB}
in the $Q$ extension.
However, the index set is generally not ${<}\lambda$-directed
any more, unless we either assume $\lambda\le\kappa$ (as in that
case there are no new small subsets of the partial order)
or $Q$ is $\lambda$-cc (as then every small set in the extension is
covered by a small set from the ground model).
\end{proof}


%

If $P$ forces strong witnesses, then any complete subforcing
that includes names for all witnesses also forces strong witnesses:

\begin{lemma}\label{lem:blasssub}
Assume that $R$ is a Borel relation, $P'$ is a complete
subforcing of $P$, $\lambda$ regular and $\mu$ is a cardinal,
both preserved in the $P$-extension.
\begin{enumerate}[(a)]
    \item If $P\Vdash\LCU_R(\lambda)$
witnessed by some $\dot{\bar{f}}$, and
 $\dot {\bar f}$ is actually a $P'$-name, then
 $P'\Vdash\LCU_R(\lambda)$.
 \item If $P\Vdash\COB_R(\lambda,\mu)$
witnessed by some $(\dot{\trianglelefteq},\dot {\bar g})$, and
$(\dot{\trianglelefteq},\dot {\bar g})$ is actually a $P'$-name, then
 $P'\Vdash\COB_R(\lambda,\mu)$.
\end{enumerate}
\end{lemma}

\begin{proof}
Let $V_2$ be the $P$-extension and
$V_1$ the intermediate $P'$-extension.
For LCU:
\eqref{eq:LCU} holds in $V_2$, $V_1\subseteq V_2$
and $(f_i)_{i<\lambda}\in V_1$,
and $R$ is absolute between $V_1$ and $V_2$,
so \eqref{eq:LCU} holds in $V_1$. 
The argument for COB is similar.
\end{proof}


We now define three properties of cardinal
characteristics
(more general than Blass-uniform)
that have implications for their behaviour
in extensions without new ${<}\kappa$-sequences.
We call these properties e.g.\
\plike\ to refer to the ``typical'' representative $\tfrak$.
But note that this is very superficial: There is no
deep connection or similarity to  $\mathfrak t$ for
all \plike\ characteristics, it is just that $\mathfrak t$
is a well known example for this property,
and ``\plike'' seems easier to memorize than
other names we came up with.
\begin{definition}
Let $\mathfrak x$ be a cardinal characteristic.
\begin{enumerate}[(1)]
\item
   $\mathfrak x$ is \emph{\plike},
   if it has the following form:
   There is a formula $\psi(x)$ (possibly with, e.g., real parameters)
   absolute between universe extensions that do not add
   reals,\footnote{Concretely,
   if $M_1\subseteq M_2$ are transitive (possibly class) models of a fixed, large
   fragment of ZFC,
   with the same reals, then
   $\psi$ is absolute between $M_1$ and $M_2$.}
   such that $\mathfrak x$ is the smallest cardinality
   $\lambda$ of a set $A$ of reals such that $\psi(A)$.

   All Blass-uniform characteristics are \plike;
   other examples are $\mathfrak t$, $\mathfrak u$, $\mathfrak a$ and $\mathfrak i$.
\item
    $\mathfrak{x}$ is called \emph{\hlike},
    if it satisfies the same,
    but with $A$ being a family of sets of reals (instead of just a set
    of reals).

    Note that \plike\ implies \hlike, as we can include ``the family of
    sets of reals is a family of singletons'' in $\psi$.
    Examples are $\mathfrak h$ and $\mathfrak g$.
\item
   $\mathfrak{x}$ is called \emph{\mlike},
   if it has the following form:
   There is a formula $\varphi$ (possibly with, e.g., real parameters)
   such that $\mathfrak x$ is the smallest cardinality
   $\lambda$  such that $H({\le}\lambda)\vDash \varphi$.

   Any infinite  \plike\ characteristic is \mlike:  
   If $\psi$ witnesses \plike, then we can use
   $\varphi=(\exists A)\, [\psi(A)\&(\forall a\in A)\ a\text{ is a real}]$
   to get \mlike\ (since $H({\le}\lambda)$ contains all reals).
   Examples are\footnote{$\mathfrak m$ can be characterized as the smallest $\lambda$ such that there is in $H({\le} \lambda)$ a ccc forcing $Q$ and a family $\bar D$ of dense
   subsets of $Q$ such
    that ``there is no filter $F\subseteq Q$ meeting all $D_i$'' holds.}
   $\mathfrak m$, $\mathfrak m(\textnormal{Knaster})$, etc.
\end{enumerate}
\end{definition}
(Actually, we do not know anything about \plike\ characteristics
in general, apart from the fact
that they are both \mlike\ and \hlike.)

\begin{lemma}\label{lem:trivial}
Let $V_1\subseteq V_2$ be models (possibly classes)
of set theory (or a sufficient fragment),
$V_2$ transitive and $V_1$ is either transitive or an elementary submodel of $H^{V_2}(\chi)$ for some large enough regular $\chi$,
such that $V_1\cap \omega^\omega = V_2\cap \omega^\omega$.
\begin{enumerate}[(a)]
    \item If $\mathfrak x$ is \hlike, then
    $V_1\vDash \mathfrak{x}=\lambda$ implies
    $V_2\vDash \mathfrak{x}\le |\lambda|$.
\end{enumerate}
In addition, whenever $\kappa$ is uncountable regular in $V_1$
and $V_1^{{<}\kappa}\cap V_2\subseteq V_1$:
\begin{enumerate}[(a)]
\setcounter{enumi}{1}
    \item
    If $\mathfrak x$ is \mlike,
    then
    $V_1\vDash \mathfrak{x}\ge\kappa$ iff
    $V_2\vDash \mathfrak{x}\ge \kappa$.
    \item If $\mathfrak x$ is \mlike\  and
    $\lambda<\kappa$, then
    $V_1\vDash \mathfrak{x}=\lambda$ iff
    $V_2\vDash \mathfrak{x}=\lambda$.
    \item If $\mathfrak x$ is \tlike\  and
    $\lambda=\kappa$, then
    $V_1\vDash \mathfrak{x}=\lambda$ implies
    $V_2\vDash \mathfrak{x}=\lambda$.
\end{enumerate}
\end{lemma}

\begin{proof}
%
First note that (d) follows by (a) and (b) because any \plike\ characteristic
is both \mlike\ and \hlike.

Assume $V_1$ is transitive.
For (a), if $\psi$ witnesses that $\mathfrak{x}$ is \hlike, $A\in V_1$ and $V_1$ satisfies $\psi(A)$, then the same holds in $V_2$. For (b) and (c), note that
$H^{V_1}(\le\mu)=H^{V_2}(\le\mu)$ for all $\mu<\kappa$
(easily shown by $\in$-induction).

The case  $V_1=N\preceq H^{V_2}(\chi)$ is similar. Note that $H^{V_2}(\chi)$ is a transitive subset of $V_2$, so (a) follows by the previous case.
For (b) and (c), work inside $V_2$. Note that $\kappa\subseteq N$ (by induction).
Whenever $\mu<\kappa$,
$\mu$ is regular iff $N\models$``$\mu\text{ regular}$'',
 and $H({\leq}\mu)\subseteq N$. 
So $N\models$``$H({\leq}\mu)\models \phi$''
iff $H({\leq}\mu)\models \phi$.
%

Alternatively, the case $V_1\preceq H^{V_2}(\chi)$ is a consequence of the first case. Work in $V_2$. Let $\pi:V_1\to\bar{V}_1$ be the transitive collapse of $V_1$. Note that $\pi(x)=x$ for any $x\in\omega^\omega\cap V_1$, so $\omega^\omega\cap\bar{V}_1=\omega^\omega\cap V_1=\omega^\omega$. To see (a), $V_1\vDash\xfrak=\lambda$ implies $\bar{V}_1\vDash \xfrak=\pi(\lambda)$, so $\xfrak\leq|\pi(\lambda)|\leq|\lambda|$ by the transitive case.

Now assume $V_1^{<\kappa}\subseteq V_1$ (still inside $V_2$), so we also have $\bar{V}_1^{<\kappa}\subseteq\bar{V}_1$. To see (b), $V_1\models\xfrak\geq\kappa$ iff $\bar{V}_1\models\xfrak\geq\pi(\kappa)=\kappa$, iff $V_2\models \xfrak\geq\kappa$ by the transitive case. Property (c) follows similarly by using $\pi(\lambda)=\lambda$ (when $\lambda<\kappa$).
\end{proof}

We apply this to three situations:
Boolean ultrapowers (which we will not apply in this paper),
extensions by distributive forcings,
and complete subforcings:

\begin{corollary}\label{cor:trivial}
Assume that $\kappa$ is uncountable regular, $P\Vdash \mathfrak{x}=\lambda$, and
\begin{enumerate}[(i)]
    \item \underline{either} $Q$ is a $P$-name for a ${<}\kappa$-distributive
    forcing, and we set $P^+:=P*Q$
    and $j(\lambda):=\lambda$;
\item \underline{or} 
$P$ is $\nu$-cc for some $\nu<\kappa$,
$j:V\to M$ is a complete embedding into a transitive ${<}\kappa$-closed model $M$, $\crit(j)\geq\kappa$, and we set $P^+:=j(P)$,
\item\underline{or} $P$ is $\kappa$-cc, $M\preceq H(\chi)$ is
${<}\kappa$-closed, and we  set $P^+:= P\cap M$ and $j(\lambda):=|\lambda\cap M|$.
(So $P^+$ is a complete subposet of $P$; and if $\lambda\le\kappa$
then $j(\lambda)=\lambda$.)
\end{enumerate}
Then we get:
\begin{enumerate}[(a)]
    \item
    If $\mathfrak x$ is \mlike\ and
    $\lambda\ge\kappa$, then
    $P^+\Vdash \mathfrak{x}\ge \kappa$.
    \item If $\mathfrak x$ is \mlike\  and
    $\lambda<\kappa$, then
    $P^+\Vdash \mathfrak{x}=\lambda$.
    \item  If $\mathfrak x$ is \hlike\ then $P^+\Vdash\mathfrak{x}\le |j(\lambda)|$. Concretely,
    \begin{itemize}
        \item[{}] for (i):
        $P^+\Vdash \mathfrak{x}\le |\lambda|$;
        \item[{}] for (ii):
        $P^+\Vdash \mathfrak{x}\le |j(\lambda)|$;
        \item[{}]
        for (iii):
        $P^+\Vdash \mathfrak{x}\le |\lambda\cap M|$.
    \end{itemize}
    \item So if $\mathfrak x$ is \tlike\ and $\lambda=\kappa$,
    then for (i) and (iii) we get $P^+\vDash \xfrak=\kappa$.
\end{enumerate}
\end{corollary}

\begin{proof}
\textbf{Case (i).} Follows
directly from Lemma~\ref{lem:trivial}.

\textbf{Case (ii).}
Since
$M$ is ${<}\kappa$-closed
and $P$ is $\nu$-cc,
$P$ (or rather: the isomorphic image $j''P$) is a complete subforcing of $j(P)$.
Let $G$ be a $j(P)$-generic filter
over $V$. As $j(P)$ is in $M$ (and
$M$ is transitive), $G$
is generic over $M$ as well.
Then $V_1:=M[G]$  is ${<}\kappa$ closed in $V_2:=V[G]$.

First note that $V_1$ and $V_2$ have the same ${<}\kappa$-sequences of ordinals.
Let $\dot{\bar x}=(\dot x_i)_{i\in\mu}$ be a sequence
of $j(P)$-names for members of $M$ with $\mu<\kappa$.
Each $\dot x_i$ is determined by an
antichain, which has size ${<}\nu$
and therefore is in $M$,
so each $\dot x_i$  is in $M$. Hence
$\dot{\bar x}$ is in $M$.

By elementaricity, $P\Vdash \mathfrak{x}=\lambda$
implies $M\models$``$j(P)\Vdash \mathfrak{x}=j(\lambda)$''.
So $V_1\models \mathfrak{x}=j(\lambda)$,
and we can apply Lemma~\ref{lem:trivial}:
In the case that $\mathfrak{x}$ is \mlike,
if $\lambda\ge\kappa$, then
$j(\lambda)\ge j(\kappa)\geq \kappa$,
so $V_2\models \mathfrak{x}\ge \kappa$;
If $\lambda<\kappa$, then $j(\lambda)=\lambda$, so
$V_2\models \mathfrak{x}=\lambda$;
if $\mathfrak{x}$ is \hlike,
then $V_2\models \mathfrak{x}\le|j(\lambda)|$.

\textbf{Case (iii).}
Let $\pi^0: M\to \bar M$ be the transitive collapse.
Set $\bar P:=\pi^0(P)\in \bar M$.
Note that $\pi^0(\kappa)=\kappa$ and that $\bar M$ is ${<}\kappa$-closed. Also, any condition in $P$ is $M$-generic since, for any antichain $A$ in $P$, $A\in M$ iff $A\subseteq M$ (by ${<}\kappa$-closedness).

Let $G^+$ be $P^+$-generic over $V$.
We can extend $G^+$ to a $P$-generic $G$ over $V$ (as $P^+$
is a complete subforcing of $P$), and we get
$G^+=G\cap P^+=G\cap M$.
Now work in $V[G]$.
Note that $M[G]$ is an elementary submodel of $H^{V[G]}(\chi)$
(and obviously not transitive),
and that the transitive collapse $\pi: M[G]\to V_1$
extends $\pi^0$ (as there are no new elements of $V$ in $M[G]$).
We claim that $V_1=\bar M[\bar G^+]$ where $\bar G^+:=\pi^0{''}G^+$ (which is $\bar P$-generic over $\bar M$, also $\bar G^+=\pi(G)$), and that
$\bar\tau[\bar G^+]=\pi(\tau[G])$ for any $P$-name $\tau\in M$, where $\bar\tau:=\pi^0(\tau)$.\footnote{This can be proved by induction on the rank of $\tau$, and uses that $M[G]\preceq H^{V[G]}(\chi)$.}
So in particular, $V_1$ is a subset of $V_2:=V[G^+]$
(the $P^+$-generic extension of $V$) because $\pi^0$ and $M$ (and therefore $\bar M$) are elements of $V$, so $G^+$ (and therefore $\bar G^+$) are elements of $V[G^+]$. In fact, $\bar G^+$ is $\bar P$-generic over $V$ because $\bar M$ is ${<}\kappa$-closed and $\bar P$ is $\kappa$-cc, moreover, $V_2=V[\bar{G}^+]$ (this is reflected by the fact that, in $V$, $\pi^0{\upharpoonright}P^+$ is an isomorphism between $P^+$ and $\bar{P}$).

We claim:
\begin{equation}\tag{$\ast$}\label{eq:jioew}
\text{$V_2$ is an NNR extension of $V_1$, moreover
$V_1$ is ${<}\kappa$-closed in $V_2$.}
\end{equation}
To show this, work in $V$. We argue with $\bar{P}$.
Let $\tau$ be a $\bar P$-name of an element of $V_1=\bar{M}[\bar G^+]$.
So we can find a maximal antichain $A$ in $\bar P$
and, for each $a\in A$, a $\bar P$-name $\sigma_{a}$ in $\bar M$
such that $a\Vdash_{\bar P} \tau=\sigma_a$.
Since $|A|<\kappa$ and $\bar P\subseteq \bar M$
and $\bar M$ is ${<}\kappa$-closed, $A$, as well as the function $a \mapsto \sigma_a$, are in $\bar M$.
Mixing the names $\sigma_a$ along $A$ to a name $\sigma\in \bar M$,
we get $\bar M \vDash a\Vdash_{\bar{P}} \sigma_a=\sigma$ for all $a\in A$, which implies $V\vDash a\Vdash_{\bar{P}} \sigma_a=\sigma$ because the forcing relation of atomic formulas is absolute.
So  $\bar{P}\Vdash\tau=\sigma$.

Now fix a $\bar P$ name $\vec\tau=(\tau_\alpha)_{\alpha<\mu}$
of a sequence of elements of $V_1$, with $\mu<\kappa$.
Again we use closure of $\bar M$ and get a
sequence $(\sigma_\alpha)_{\alpha<\mu}$ in $\bar M$
such that $\bar P$ forces that $\tau_\alpha=\sigma_\alpha[\bar G^+]$, and so
the evaluation of  the sequence $\vec \tau$ is in $\bar M[\bar G^+]=V_1$.
This proves~(\ref{eq:jioew}).

Now assume that $\xfrak$ is either \hlike\ or \mlike, and $P\Vdash \mathfrak x= \lambda$.
By elementaricity, this holds in $M$, so $\bar M\vDash \bar P\Vdash \xfrak=\pi^0(\lambda)$. Now let $\bar G^+$ be $\bar P$-generic over $V$, $V_1:=\bar M[\bar G^+]$ and $V_2:= V[\bar G^+]$, so $V_1\models\xfrak=\pi^0(\lambda)$. If $\xfrak$ is \hlike\ then, by Lemma~\ref{lem:trivial}(a), $V_2\models\xfrak\leq|\pi^0(\lambda)|=|\lambda\cap M|$; if $\xfrak$ is \mlike\ and $\lambda<\kappa$, then $V_1\models\xfrak=\lambda$ and so the same is satisfied in $V_2$ by Lemma~\ref{lem:trivial}(c); otherwise, if $\lambda\geq\kappa$ then $V_1\models\xfrak=\pi^0(\lambda)\geq\pi^0(\kappa)=\kappa$, so $V_2\models \xfrak\geq\kappa$ by Lemma~\ref{lem:trivial}(b).

In any of the cases above, (d) is a direct consequence of (a) and (c).
\end{proof}

\section{Dealing with \texorpdfstring{$\mfrak$}{m}}\label{sec:ma}

We show how to deal with $\mathfrak m$.
It is easy to check that the Cicho\'n's Maximum construction from \cite{GKS}
forces $\mathfrak m=\aleph_1$,
and can easily be
modified to force $\mathfrak m=\addN$
(by forcing with all
small ccc forcings during the iteration).
With a bit more work
it is also possible to get $\aleph_1<\mfrak <\addN$.

Let us start by recalling the definitions of some well-known classes of ccc forcings:
\begin{definition}\label{DefKnaster}
    Let $\lambda$ be an infinite cardinal, $k\ge 2$ and let $Q$ be a poset.
    \begin{enumerate}
     \item $Q$ is \emph{$(\lambda, k)$-Knaster} if, for every
    $A\in [Q]^{\lambda}$, there is a $B\in [A]^\lambda$ which is
    $k$-linked (i.e., every $c\in [B]^k$ has a lower bound in $Q$).
    We write \emph{$k$-Knaster} for ($\aleph_1,k)$-Knaster;
    \emph{Knaster} means $2$-Knaster; \emph{$(\lambda,1)$-Knaster} denotes $\lambda$-cc
    and \emph{$1$-Knaster} denotes ccc.\footnote{This is just an abuse of notation that turns out to be convenient for stating our results.}

    \item \emph{$Q$ has precaliber $\lambda$} if, for every
    $A\in [Q]^{\lambda}$, there is a $B\in [A]^{\lambda}$ which is
    centered,
    i.e., every  finite subset of  $B$ has a lower bound in $Q$. We sometimes shorten ``precaliber $\aleph_1$'' to ``precaliber''.

    \item\label{sigklink} $Q$ is \emph{$(\sigma, k)$-linked} if there is a function $\pi:Q\to \omega$
    such that $\pi^{-1}(\{n\})$ is $k$-linked for each $n$.

    \item $Q$ is \emph{$\sigma$-centered} if there is a function $\pi:Q\to \omega$
    such that each $\pi^{-1}(\{n\})$  is centered.
    \end{enumerate}
\end{definition}

\begin{figure}
\[
\xymatrix@=4.5ex{
                & \textnormal{$\sigma$-$2$-linked}\ar[d] & \textnormal{$\sigma$-$3$-linked}\ar[l]\ar[d] & \cdots\ar[l]\ar[d] & \sigma\textnormal{-centered}\ar[d]\ar[l]\\
\textnormal{ccc} & 2\textnormal{-Knaster}\ar[l] & 3\textnormal{-Knaster}\ar[l] & \cdots\ar[l] & \textnormal{precaliber}\ar[l]
}
\]
\caption{Some classes of ccc forcings\label{fig:cccetal}}
\end{figure}
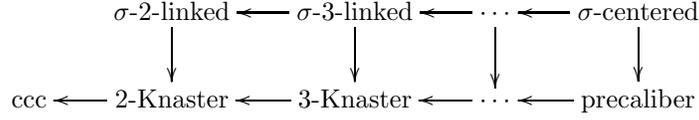

The implications between these notions (for $\lambda=\aleph_1$)
are listed in Figure~\ref{fig:cccetal}.
To each class $C$ of forcing notions, we can define the Martin's Axiom number $\mfrak(C)$ in the usual way (recall Definition \ref{DefChar1}). An implication $C_1\leftarrow C_2$
in the diagram corresponds to a ZFC inequality $\mfrak(C_1)\le \mfrak(C_2)$.
Recall that $\mathfrak m(\sigma\textrm{-centered})=\mathfrak p = \mathfrak t$. Also recall that, in the old constructions,
all iterands were $(\sigma,k)$-linked for all $k$.

\begin{lemma}\label{lem:suslin}
\begin{enumerate}
  \item\label{item:lekwjt0}
    If there is a Suslin tree, then  $\mfrak=\aleph_1$.
  \item\label{item:lekwjt1}
    After adding a Cohen real $c$ over $V$, in $V[c]$ there is a
    Suslin tree.
  \item\label{item:lekwjt2}
    Any Knaster poset preserves Suslin trees.
  \item\label{item:lekwjt3}
  The result of any finite support iteration of  $(\lambda,k)$-Knaster posets
   ($\lambda$ uncountable regular and $k\ge 1$) is again $(\lambda,k)$-Knaster.
   \item\label{item:lekwjt4}
  In particular, when $k\geq 1$,
  if $P$ is a f.s.\ iteration of forcings such that all iterands
  are either $(\sigma, k)$-linked or smaller than $\lambda$,
  then $P$ is $(\lambda,k)$-Knaster.
    \item\label{item:martinsaysknown}
    Let $C$ be any of the forcing classes
    of Figure~\ref{fig:cccetal},
    and assume $\mfrak(C)=\lambda>\aleph_1$.
    \\
    (Or just assume that $C$ is a class of ccc forcings
    closed under $Q\mapsto Q^{<\omega}$, the
    finite support product of countably many copies of $Q$,
    and under $(Q,p)\mapsto \{q: q\le p\}$ for $p\in Q$.)

    If $Q\in C$,
    then every subset $A$ of $Q$ of size ${<}\lambda$
    is ``$\sigma$-centered in $Q$'' (i.e., there is a
    function $\pi:A\to\omega$ such that every finite $\pi$-homogeneous
    subset of $A$ has a common lower bound in $Q$).

    So in particular, for all $\mu<\lambda$ of uncountable cofinality,
    $Q$ has precaliber $\mu$
    and is $(\mu,\ell)$-Knaster for all $\ell\ge 2$.

    \item\label{item:corcor} $\mfrak>\aleph_1$ implies $\mfrak=\mfrak(\textnormal{precaliber})$.
    \\
    $\mfrak(k\textnormal{-Knaster})>\aleph_1$ implies
    $\mfrak(k\textnormal{-Knaster})=\mfrak(\textnormal{precaliber})$.
\end{enumerate}
\end{lemma}

\begin{proof}
(\ref{item:lekwjt0}): Clear.
(\ref{item:lekwjt1}): See~\cite{Sh:176, To93} or Velleman~\cite{MR763903}.
(\ref{item:lekwjt2}): Recall that the product of a Knaster poset with a ccc poset is still ccc. Hence, if $P$ is Knaster and $T$ is a Suslin tree, then $P\times T=P\ast\check{T}$ is ccc, i.e., $T$ remains Suslin in the $P$-extension.

(\ref{item:lekwjt3}): Well-known, see e.g.\ Kunen~\cite[Lemma V.4.10]{Kunen} for $(\aleph_1,2)$-Knaster.
The proof for the general case is the same, see e.g.~\cite[Section 5]{mejiavert}.

(\ref{item:lekwjt4}): Clear, as $(\sigma,k)$-linked implies
$(\mu,k)$-Knaster (for all uncountable regular $\mu$),
and since every forcing of size ${<}\mu$ is $(\mu,k)$-Knaster (for any $k$).

(\ref{item:martinsaysknown}):
%
First note that it is well known\footnote{See, e.g., Jech~\cite[Thm.~16.21]{Je03} (and the historical remarks, where the result is attributed to (independently) Kunen, Rowbottom and Solovay), or~\cite[Lem.~1.4.14]{BaJu} or Galvin~\cite[Pg. 34]{Galvin}.}
that $\text{MA}_{\aleph_1}$(ccc) implies that every ccc forcing is Knaster, and hence that the class $C$ of ccc forcings is closed under $Q\mapsto Q^{<\omega}$.
(For the other classes $C$ in Figure~\ref{fig:cccetal}, the closure
is immediate.)

So let $C$ be a closed class, $\mfrak(C)=\lambda>\aleph_1$,  $Q\in C$ and $A\in [Q]^{<\lambda}$.
Given a filter $G$ in $Q^{<\omega}$ and $q\in Q$,
    set $c(q)=n$ iff $n$ is minimal such that there is a $\bar p\in G$
    with $p(n)=q$. Note that for all $q$, the set
    \[D_q=\{p\in Q^{<\omega}:\, (\exists n\in\omega)\, q=p(n)\}\]
    is dense, and that $c(q)$ is defined whenever $G$ intersects $D_q$. Pick a filter $G$ meeting all
    $D_q$ for $q\in A$. This defines $c:A\to \omega$ such that
    $c(a_0)=c(a_1)=\dots = c(a_{\ell-1})=n$ implies that
    all $a_i$ appear in $G(n)$ and thus they are compatible in $Q$. Hence, $A$ is the union of countably many centered (in $Q$) subsets of $Q$.

(\ref{item:corcor}): Follows as a corollary.
\end{proof}

This shows that
it is not possible to simultaneously separate
more than two Knaster numbers.
More specifically:
ZFC proves that there is a (unique)
$1\le k^*\le\omega$ and, if $k^*<\omega$, a (unique) $\lambda>\aleph_1$,
such that for all $1\le\ell<\omega$
\begin{equation}\label{eq:knisterknaster}
\mfrak(\ell\textnormal{-Knaster})=
\begin{cases}
\aleph_1&\text{if }\ell<k^* \\
\lambda &\text{otherwise}.
\end{cases}
\end{equation}
(Recall that $\mfrak(1\textnormal{-Knaster}) = \mfrak(\textnormal{ccc})$ by our definition.)
So the case $k^*=\omega$ means that all Knaster numbers are $\aleph_1$.

In this section, we will show how these constellations
can be realized together with the previous values for the Cicho\'n-characteristics.

In the case
$k^*<\omega$, we
know that $\mfrak(\textnormal{precaliber})=\lambda$ as well. We briefly comment that
$\mfrak(\textnormal{precaliber})=\aleph_1$
(in connection with the Cicho\'n-values) is possible too.
In the next section, we will deal with the
remaining case:
$k^*=\omega$,
i.e., all Knaster numbers are $\aleph_1$, while
$\mfrak(\textnormal{precaliber})>\aleph_1$.

The central observation is the following, see~\cite{TodorCell,To93} and~\cite[Sect.~3]{Barnett}.

\begin{lemma}\label{lem:blamain}
Let $k\in\omega$, $k\ge 2$ and $\lambda$ be
uncountable regular.
Let $C$ be the finite support iteration of
$\lambda$ many copies of Cohen forcing.
Assume that $C$ forces that $P$ is $(\lambda,k+1)$-Knaster.
Then $C*P$ forces $\mfrak(k\textnormal{-Knaster})\le \lambda$.

The same holds for $k=1$ and $\lambda=\aleph_1$.
\end{lemma}

For $k=1$ this trivially follows from
Lemma~\ref{lem:suslin}:
The first Cohen forcing adds a
Suslin tree, which is preserved by the rest of the Cohen posets composed
with $P$. So we get $\mfrak=\aleph_1$.
The proof for $k>1$ is done in the following two lemmas.

\begin{remark}
   Adding the Cohen reals first is just for notational convenience.
   The same holds, e.g., in a f.s.\ iteration where we
   add Cohen reals on a subset of the index set of order type $\lambda$;
   and we assume that the (limit of the) whole iteration is
   $(\lambda,k+1)$-Knaster.
\end{remark}

\begin{lemma}\label{lem:help1}
Under the assumption of Lemma~\ref{lem:blamain}, for $k\ge 1$:
We interpret each Cohen real
$\eta_\alpha$ ($\alpha\in\lambda$) as an element of
$(k+1)^\omega$. $C*P$ forces: For all $X\in [\lambda]^\lambda$,
\begin{equation}\label{eq:uheqt}\tag{$**$}
    (\exists \nu\in (k+1)^{<\omega})\,
    (\exists \alpha_0,\dots,\alpha_k\in X)\, (\forall 0\le i\le k) \nu^\frown i \vartriangleleft \eta_{\alpha_i}
\end{equation}
\end{lemma}

\begin{proof}
Let $p^*\in C\ast P$ force that $X\in[\lambda]^\lambda$. By our assumption, first note that $p^*{\upharpoonright}\lambda$ forces that there is some $X'\in[\lambda]^\lambda$ and a $k+1$-linked set $\{ r_\alpha:\alpha\in X'\}$ of conditions in $P$ below $p^*(\lambda)$ such that $r_\alpha\Vdash_P\alpha\in X$ for any $\alpha\in X'$.

Since $X'$ is a $C$-name, there is some $Y\in[\lambda]^{\lambda}$
and, for each $\alpha\in Y$, some $p_\alpha\le p^*{\upharpoonright}\lambda$ in $C$
forcing $\alpha\in X'$.
We can assume that $\alpha\in\dom(p_\alpha)$ and, by thinning out $Y$,
that
$\dom(p_\alpha)$ forms a $\Delta$-system with heart $a$
below each $\alpha\in Y$, $\la p_\alpha{\upharpoonright}a:\alpha\in Y\ra$ is constant, and that
$p_\alpha(\alpha)$ is always the same Cohen condition $\nu\in (k+1)^{<\omega}$.

For each $\alpha\in Y$ let $q_\alpha\in C\ast P$ such that $q_\alpha{\upharpoonright}\lambda=p_\alpha$ and $q_\alpha(\lambda)=r_\alpha$. It is clear that
$\la q_\alpha : \alpha\in Y\ra$ is $k+1$-linked and that $q_\alpha\Vdash\alpha\in X$.
Pick $\alpha_0,\dots,\alpha_k\in Y$ and
$q\le q_{\alpha_0},\dots, q_{\alpha_k}$.
We can assume that $q{\upharpoonright}\lambda$ is just the union
of the $q_{\alpha_i}{\upharpoonright}\lambda$.
In particular, we can extend
$q(\alpha_i)=\nu$ to $\nu^\frown i$, satisfying~\eqref{eq:uheqt} after all. This proves the claim.
\end{proof}

\begin{lemma}\label{lem:blauwz}
Under the assumption of Lemma~\ref{lem:blamain}, for $k\ge 2$: In $V^C$
define $R_{K,k}$ to be the set of finite partial
functions $p:u\to \omega$, $u\subseteq \lambda$ finite, such
that~\eqref{eq:uheqt} fails for all $p$-homogeneous $X\subseteq u$.\footnote{Say that $X\subseteq u$ is \emph{$p$-homogeneous} if $p{\upharpoonright}X$ is a constant function.}
Then $P$ forces the following:
\begin{enumerate}[(a)]
    \item There is no filter on $R_{K,k}$ meeting all
    dense $D_\alpha$ ($\alpha\in\lambda$), where we set
    $D_\alpha=\{p: \alpha\in \dom(u)\}$.
    \item $R_{K,k}$ is $k$-Knaster.
\end{enumerate}
\end{lemma}
Note that this proves Lemma~\ref{lem:blamain}, as $R_{K,k}$
is a witness.

\begin{proof}
Clearly each $D_\alpha$ is dense (as we can just use a hitherto unused color).
If $G$ is a filter meeting all $D_\alpha$,
then
$G$ defines a total function $p^*:\lambda\to\omega$,
and there is some $n\in\omega$ such that
$X:={p^*}^{-1}(\{n\})$ has size $\lambda$.
So~\eqref{eq:uheqt} holds for $X$, witnessed by some
$\alpha_0,\dots,\alpha_k$.
Now pick some $q\in G$ such that all $\alpha_i$ are
in the domain of $q$. Then $q$ contradicts the definition
of $R_{K,k}$.

$R_{K,k}$ is $k$-Knaster:
Given $(r_\alpha:u_\alpha\to\omega)_{\alpha\in\omega_1}$,
we thin out so that $u_\alpha$ forms a $\Delta$-system of sets of the same size
and such that each $r_\alpha$ has the same ``type'', independent
of $\alpha$, where the type contains the following information:
The color assigned to the $n$-the element of $u_\alpha$;
the (minimal, say) $h$ such that all $\eta_\beta\restriction h$
are distinct for $\beta\in u_\alpha$, and
$\eta_\beta\restriction h+1$.

We claim that the union of $k$ many
such $r_\alpha$ is still in $R_{K,k}$:
Assume towards a contradiction that there is a $\bigcup_{i<k}r_i$-homogeneous set $\alpha_0,\dots,\alpha_k$
in $\bigcup_{i<k} u_i$ such that \eqref{eq:uheqt} holds for $\nu\in(k+1)^{H}$ for some $H\in\omega$. Assume $H\ge h$.
Note that $\eta_\beta\restriction h$ are already distinct for
the different $\beta$ in the same $u_i$, so all $k+1$ many
$\alpha_j$ have to be the $n^*$-th element of different $u_i$
($n^*$ fixed), which is impossible as there are only $k$ many $u_i$.
So assume $H<h$. But then $\eta_\beta\restriction H+1$
and the color of $\beta$ both are determined by the position
of $\beta$ within $u_i$; so without loss of generality all
the $\alpha_j$ are in the same $u_i$, which is impossible
as $r_i:u_i\to \omega$ was a valid condition.

To summarize: $P$ forces
that there is a $k$-Knaster poset
$R_{K,k}$ and $\lambda$ many dense sets not met by
any filter. Therefore $P$ forces that $\mfrak(k\textnormal{-Knaster})\le
\lambda$.
\end{proof}


Let $\Ppre$ be the initial forcing of Lemma~\ref{lem:oldleft};
recall that it forces $\addN=\nu_1$
and $\bfrak=\nu_3$.
\begin{lemma}\label{lem:eleven}
For each of the following items (1) to (3),
and $\aleph_1\le\lambda\leq\nu_1$ regular,
$\Ppre$
can be modified to some forcing $P'$
which still strongly witnesses the
Cicho\'n-characteristics, and additionally satisfies:
\begin{enumerate}
  \item\label{item:hwfsm}
    Each iterand in $P'$ is $(\sigma,\ell)$-linked for all $\ell\ge 2$; and
    $P'$ forces
    \[
      \aleph_1=\mfrak=\mfrak(\textnormal{precaliber})\le\mathfrak p=\mathfrak b.
    \]
  \item\label{item:hwfsm2}
    Fix $k\ge 1$. Each iterand in $P'$ is $k+1$-Knaster, and additionally
    either $(\sigma,\ell)$-linked for all $\ell$ or of size less than $\lambda$; and
    $P'$ forces
    \[\aleph_1=\mfrak = \mfrak(k\textnormal{-Knaster})<
    \mfrak(k+1\textnormal{-Knaster})=\mfrak(\textnormal{precaliber})=\lambda \leq\mathfrak p=\mathfrak b.
    \]
    \item\label{item:hwflgbla}
    Each iterand in $P'$ is either
    $(\sigma, \ell)$-linked for all $\ell$, or ccc of size less than $\lambda$; and
    $P'$ forces
    \[
    \mfrak=\mfrak(\textnormal{precaliber})=\lambda\leq\mathfrak p=\mathfrak b.
    \]
\end{enumerate}
\end{lemma}

\begin{proof}
An argument like in \cite{Br} works.
We first modify $\Ppre$ as follows:

We construct an iteration $P$ with the same index set $\delta$
as $\Ppre$; we partition $\delta$ into two cofinal sets
$\delta=S_{\textrm{old}} \cup S_{\textrm{new}}$ of the same size.
For $\alpha\in S_{\textrm{old}}$  we define $Q_\alpha$
as we defined $Q^*_\alpha$ for $\Ppre$.
For $\alpha\in S_{\textrm{new}}$, pick
(by suitable book-keeping) a small (less than $\nu_3$, the value for $\bfrak$)
$\sigma$-centered forcing $Q_\alpha$.

As $\cf(\delta)\ge \lambda_{\mathfrak b}$, we get that $P$ forces $\mathfrak p\ge \nu_3$. Also, $P$ still adds strong witnesses for the
Cicho\'n-characteristics, according to Claim~\ref{claim:oldtonew}:
All new iterands are smaller than $\nu_3$
and $\sigma$-centered.


Note that all iterands are still $(\sigma,k)$-linked for all $k$
(as the new ones are even $\sigma$-centered).\smallskip

To deal with $\ell$-Knaster, recall that
the first $\lambda_\infty$ iterands are Cohen forcings;
and we call these Cohen reals $\eta_\alpha$
($\alpha\in\lambda_\infty$).
Given $\ell$, we can (and will) interpret the Cohen real $\eta_\alpha$ as an element of $(\ell+1)^\omega$.\smallskip

\noindent(\ref{item:hwfsm})
Recall from \cite[Sect. 2]{Barnett} that, after a Cohen real, there is a precaliber $\omega_1$ poset $Q^*$ such that no $\sigma$-linked poset adds a filter intersecting certain $\aleph_1$-many dense subsets of $Q^*$.\footnote{To be more precise, after one Cohen real there is a sequence $\bar{r}=\la r_\alpha:\omega\to 2:\alpha\in\omega_1\text{\ limit}\ra$ such that, for any ladder system $\bar{c}$ from the ground model, the pair $(\bar{c},\bar r)$, as a ladder system coloring, cannot be uniformized in any stationary subset of $\omega_1$. Furthermore, this property is preserved after any $\sigma$-linked poset. Also recall from \cite{DevS} (with Devlin) that $\mfrak(\textnormal{precaliber})>\aleph_1$ implies that any ladder system coloring can be uniformized.} Therefore, the $P$ we just constructed forces $\mfrak(\textnormal{precaliber})=\aleph_1$.

\noindent(\ref{item:hwfsm2}) Just as with the modification from $\Ppre$ to $P$,
we now further modify $P$ to force (by some bookkeeping)
with all small (smaller than $\lambda$)
$k+1$-Knaster forcings.
So the resulting iteration obviously forces $\mfrak(k+1\textnormal{-Knaster})\ge \lambda$.

Note that now all iterands are either smaller than $\lambda\le\nu_1$
or $\sigma$-linked (so we can again use  Claim~\ref{claim:oldtonew}); and additionally all iterands are $k+1$-Knaster.
So $P$ is both $(\aleph_1,k+1)$-Knaster and
$(\lambda,\ell)$-Knaster for any $\ell$.
Again by Lemma~\ref{lem:blamain},
$P$ forces both
$\mfrak(k\textnormal{-Knaster})=\aleph_1$ and
$\mfrak(\ell\textnormal{-Knaster})\le \lambda$ for any $\ell$,
(which implies $\mfrak(k+1\textnormal{-Knaster})=\lambda$).
\smallskip

\noindent(\ref{item:hwflgbla})
This is very similar, but this time we use all small
ccc forcings (not just the $k+1$-Knaster ones).
This obviously results in $\mfrak\ge\lambda$;
and the same argument as above shows that still
$\mfrak(\ell\textnormal{-Knaster})\le \lambda$ for all $\ell$.
\end{proof}

This, together with Corollary~\ref{cor:trivial}
and Lemma~\ref{lem:oldnew} gives us
11 characteristics.
However, we postpone this collorally until a time we can also
add $\hfrak=\gfrak=\pfrak=\kappa$ in Lemma~\ref{lem:twelve}.

\section{Dealing with the precaliber number}\label{sec:precaliber}


Recall the possible constellations for the Knaster numbers and the definition of $k^*$ and $\lambda$
given in~\eqref{eq:knisterknaster}.
(And recall that if $k^*<\omega$, i.e., then
$\mfrak(\textnormal{precaliber})=\lambda$ as well.)

In this section, we construct models for
$k^*=\omega$, i.e., all Knaster numbers being $\aleph_1$, while
$\mfrak(\textnormal{precaliber})=\lambda$ for some given
regular $\aleph_1<\lambda\le \addN$
(and for the ``old''
values for the Cicho\'n-characteristics,
as in the previous section).

\begin{definition}\label{def:51} Let $\lambda>\aleph_1$ be regular.
   A condition $p\in P_{\precal}=P_{\precal,\lambda}$ consists of
   \begin{itemize}
       \item[(i)] finite sets $u_p, F_p\subseteq\lambda$,
       \item[(ii)] a function $c_p:[u_p]^2\to 2$,
       \item[(iii)] for each $\alpha\in F_p$, a function $d_{p,\alpha}:\mathcal{P}(u_p\cap\alpha)\to\omega$ satisfying
       \begin{itemize}
           \item[$(\star)$] if $\alpha\in F_p$ and
           $s_1$, $s_2$ are
       1-homogeneous (w.r.t.\ $c_p$)\footnote{Say that $s\subseteq u_p\cap \alpha$ is \emph{$1$-homogeneous w.r.t.\ $c_p$} if $c_p(\xi,\zeta)=1$ for any $\xi\neq\zeta$ in $s$.} subsets of $u_p\cap \alpha$ with
       $d_{p,\alpha}(s_1)=d_{p,\alpha}(s_2)$, then $s_1\cup s_2$ is 1-homogeneous.
       \end{itemize}
   \end{itemize}
   The order is defined by $q\leq p$ iff $u_p\subseteq u_q$, $F_p\subseteq F_q$, $c_p\subseteq c_q$ and $d_{p,\alpha}\subseteq d_{q,\alpha}$ for any $\alpha\in F_p$.
\end{definition}

\begin{lemma}\label{forcecolor}
   $P_{\precal}$ has precaliber $\omega_1$ (and in fact precaliber $\mu$
   for any regular uncountable $\mu$) and forces the following:
   \begin{enumerate}
      \item The generic functions $c:[\lambda]^2\to\{0,1\}$ and $d_\alpha:[\alpha]^{<\aleph_0}\to\omega$ for $\alpha<\lambda$ are totally defined.
       \item Whenever $(s_i)_{i\in I}$
       is a family of
       finite, 1-homogeneous (w.r.t.\ $c$) subsets of~$\alpha$,
       and $d_\alpha(s_i)=d_\alpha(s_j)$ for $i,j\in I$, then $\bigcup_{i\in I} s_i$ is 1-homogeneous.
       \item If $A\subseteq[\lambda]^{<\aleph_0}$ is a family of size $\lambda$ of pairwise disjoint sets, then there are two sets $u\ne v$ in $A$ such that $c(\xi,\eta)=0$ for any $\xi\in u$ and $\eta \in v$.
       \item Whenever $u\in[\lambda]^{<\aleph_0}$, the set $\{\eta<\lambda : \forall \xi\in u (c(\xi,\eta)=1)\}$ is unbounded in~$\lambda$.
   \end{enumerate}
\end{lemma}

\begin{proof}
    \emph{For any $\alpha<\lambda$, the set of conditions $p\in P_{\precal}$ such that $\alpha\in F_p$ is dense.}

    Starting with $p$ such that $\alpha\notin F_p$,
    we set $u_q=u_p$, $F_q=F_p\cup\{\alpha\}$,
    and we pick new and unique values for all $d_{q,\alpha}(s)$
    for $s\subseteq u_q\cap\alpha=u_p\cap\alpha$, as well as new and unique values for all $d_{q,\beta}(s)$ for $s\subseteq u_q\cap\beta$ with $\alpha\in s$. We have to show that $q\in P_{\precal}$,
    i.e., that it satisfies $(\star)$:
    Whenever $s_1,s_2$ satisfy the assumptions of
    $(\star)$, then $\alpha\notin s_i$
    (for $i=1,2$), as we would otherwise have chosen different
    values. So we can use that $(\star)$ holds for $p$.\smallskip

    \noindent\textbf{(1) and (4)}
    \emph{For any $\xi<\lambda$, the set of $q\in P_{\precal}$ such that $\xi\in u_q$ is dense.}

    Starting with $p$ with $\xi\notin u_p$, we set $u_q=u_p\cup\{\xi\}$ and $F_q=F_p$.
    Again, pick new (and different) values for all $d_{q,\alpha}(s)$
    with $\xi\in s$,
    and we can set $c(x,\xi)$ to whatever we want.
    The same argument as above shows that $q\in P_{\precal}$.
    In particular we can set all $c(x,\xi)=1$,
    which shows that $P_{\precal}$ forces (4).\smallskip

    \noindent\textbf{(2)} follows from $(\star)$ for $I=\{1,2\}$,
    and this trivially implies the case for arbitrary $I$.
    (For $x_1,x_2\in \bigcup_{i\in I} s_i$, pick $i_1,i_2\in I$
    such that $x_1\in s_{i_1}$ and $x_2\in s_{i_2}$;
    then apply $(\star)$ to $\{i_1,i_2\}$.)\smallskip

   \noindent\textbf{Precaliber.} \emph{$P_{\precal}$ has precaliber $\mu$ for any uncountable regular $\mu$.}

   Let $A\subseteq P_\precal$ have size $\mu$.
   We can assume that the $u_p$'s and $F_p$'s for $p\in A$ form $\Delta$-systems with roots $u$ and $F$ respectively, and we can assume that all $p\in A $ have the same  ``type (over $u$, $F$)'', which is
   defined as follows:



   Let $i$ be the order-preserving bijection (Mostowski's collapse) of $u_p\cup F_p$ to some  $N\in \omega$.
 This induces sets $\bar u\subseteq \bar u_p\subseteq N$ and  $\bar F\subseteq \bar F_p \subseteq N$ and
 partial functions $\bar c_p:[\bar u_p]^2\to 2$, $\bar d_{ p,\bar\alpha}: \mathcal{P}(\bar u_p\cap\bar\alpha)\to\omega$
 (for $\bar\alpha\in \bar F_p$), such that $i$ is an isomorphism between the structures $p=(u_p\cup F_p, u_p, F_p, u, F, c_p, (d_{p,\alpha})_{\alpha\in F_p})$
 and $\bar p=(N,\bar u_p,\bar  F_p, \bar u, \bar F,\bar c_p, (\bar d_{p,\bar\alpha})_{\bar\alpha\in \bar F_p})$; the latter structure is called the \emph{type} of $p$ (over $u$, $F$).

 Let us note some trivial facts:
 There are only countably many different types;
 between any two conditions with same type there
 is a natural isomorphism; and
   if $p$ and $q$ have the same type
   (over $u=u_p\cap u_q$ and $F=F_p\cap F_q$), then
   $c_p$ and $c_q$ agree on the common domain, and the
   same holds for $d_p$ and $d_q$.\footnote{I.e., for $\alpha<\beta$ in $u$,
   $c_p(\alpha,\beta)=c_q(\alpha,\beta)$, and for $\alpha\in F$
   and $s\subseteq u$, $d_{p,\alpha}(s)=d_{q,\alpha}(s)$.}


   To summarize: Given $A\subseteq P_\precal$ of size $\mu$,
   we can find a $\mu$-sized subset $B$
   forming a $\Delta$-system such that
   all elements have the same type (over the root).
   We claim that then any finite subset
   of $B$
   has a common lower bound $q$ (which implies that
   $P_\precal$ has precaliber $\mu$, as required). This is done by amalgamation, as follows:

   \smallskip

   \noindent\textbf{Amalgamation.}
   Fix $p_0,\dots,p_{n-1}$ of the same type
   over $(u,F)$, such that $u_i\cap u_j=u$ and
   $F_i\cap F_j=F$ for all $i,j$ in $n$ (where we set
   $u_i:=u_{p_i}$ etc.).
   We define an ``amalgam'' $q$
   of these conditions
   as follows: $u_q:=\bigcup_{i\in n} u_i$,
   $F_q:=\bigcup_{i\in n} F_i$, $d_q$ extends all $d_i$
   and has a unique new value for each new element in its domain,
   $c_q$ extends all $c_i$; and yet undefined
   $c_q(x,y)$ are set to
   $0$ if $x,y>\max(F)$ (and $1$ otherwise).

   To see that $q \in P_{\precal}$, assume that $\alpha\in F_q$ and
   $s_1,s_2$ are as in $(\star)$ of Definition~\ref{def:51}. This implies that
   $d_{q,\alpha}(s_k)$ for both $k=1,2$ were already defined\footnote{By which we mean $\alpha\in F_i$ and
   $s_k\subseteq u_i$ for both $k=1,2$.}
   by one of the $p_i$ (for $i\in n$),
   otherwise we would have picked a new value.

   If they are both defined by the same $p_i$,
   we can use $(\star)$ for $p_i$. So assume otherwise,
   and for notational simplicity assume that $s_i$ is defined by
   $p_i$; and let $x_i\in s_i$. We have to show $c_q(x_1,x_2)=1$.
   Note that $\alpha\in F_1\cap F_2 = F$.
   If $x_1$ or $x_2$ are not in $u$,
   then we have set $c_q(x_1,x_2)$ to 1 (as $x_i<\alpha\in F$),
   so we are done. So assume $x_1,x_2\in u$.
   The natural isomorphism between
   $p_1$ and $p_2$
   maps $s_1$ onto some $s'_1\subseteq u_2$, and we get
   that $s'_1$ is 1-homogeneous and that
   $d_{2,\alpha}(s_2)=d_{1,\alpha}(s_1)=d_{2,\alpha}(s'_1)$.
   So we use that $p_2$ satisfies $(\star)$ to get
   that $c_2(a,b)=1$ for all $a\in s_2$ and  $b\in s'_1$.
   As the isomorphism does not move $x_1$, we can use
   $a=x_2$ and $b=x_1$.

   \smallskip

   \noindent\textbf{(3)}
   Let $p\in P_{\precal}$ and assume that $p$ forces that $\dot{A}\subseteq[\lambda]^{<\aleph_0}$ is a family of size $\lambda$ of pairwise disjoint sets. We can find, in the ground model, a family $A'\subseteq[\lambda]^{<\aleph_0}$ of size $\lambda$ and conditions $p_v\leq p$ for $v\in A'$ such that $v\subseteq u_{p_v}$, and $p_v$ forces $v\in\dot{A}$.
   We again thin out to a $\Delta$-system as above;
   this time we can additionally assume that the heart
   of the $F_v$ is below the non-heart parts of all $u_v$, i.e.,
   that $\max(F)$ is below $u_v\setminus u$ for all~$v$.

   Pick any two  $p_v$, $p_w$ in this $\Delta$-system,
   and let $q$ be the amalgam defined above.
   Then $q$ witnesses that $p_v,p_w$ are compatible, which
   implies $v\cap w=0$, i.e., $v,w$ are outside the heart;
   which by construction of $q$
   implies that $c_q$ is constantly zero on
   $v\times w$ (as their elements are above $\max(F)$).
\end{proof}

The poset $P_{\precal,\lambda}$ adds
generic functions $c$ and $d_\alpha$. We now use them to define
a precaliber $\omega_1$ poset $Q_{\precal}$ witnessing $\mfrak(\textnormal{precaliber})\le\lambda$:

\begin{lemma}\label{colorposet}
   In $V^{P_{\precal}}$, define the poset $Q_{\precal}:=\{u\in[\lambda]^{<\aleph_0} : u \text{\ is 1-homogeneous}\}$, ordered by $\supseteq$ (By 1-homogeneous,
   we mean 1-homogeneous with respect to $c$.)
   Then the following is satisfied (in $V^{P_{\precal}}$):
   \begin{enumerate}
       \item\label{item:a} $Q_{\precal}$ is an increasing union of length $\lambda$ of centered sets (so in particular it has precaliber $\aleph_1$).
       \item\label{item:b} For $\alpha<\lambda$, the set $D_\alpha:=\{u\in Q_{\precal} : u\nsubseteq\alpha\}$ is open dense. So $Q_{\precal}$ adds a cofinal generic 1-homogeneous subset of $\lambda$.
       \item\label{item:c}
       There is no 1-homogeneous set of size $\lambda$ (in $V^{P_{\precal}}$).
       In other words, there is no filter meeting all $D_\alpha$.
   \end{enumerate}
\end{lemma}
\begin{proof}
   For~(\ref{item:a}) set $Q^\alpha_{\precal}=Q_{\precal}\cap [\alpha]^{<\aleph_0}$.
   Then $d_\alpha:Q^\alpha_{\precal}\to \omega$ is a centering function,
   according to Lemma~\ref{forcecolor}(2). Precaliber $\aleph_1$ is a consequence of $\lambda_\precal>\aleph_1$.


   Property~(\ref{item:b}) is a direct consequence of Lemma~\ref{forcecolor}(4), and~(\ref{item:c}) follows
   from Lemma~\ref{forcecolor}(3).
\end{proof}

This shows that $P_{\precal,\lambda}\Vdash \mfrak(\textnormal{precaliber}) \le \lambda$. We now show
that this is preserved in further Knaster extensions.

\begin{lemma}\label{lem:blubb43}
   In $V^{P_{\precal}}$, assume that $P'$ is a ccc $\lambda$-Knaster poset.
   Then, in $V^{P_{\precal}\ast P'}$,
   $\mfrak(\textnormal{precaliber}) \le \lambda$.
\end{lemma}
\begin{proof}
   We claim that in $V^{P_{\precal}\ast P'}$,
   $Q_{\precal}$ still has precaliber $\aleph_1$, and there is no filter meeting each open dense subset $D_\alpha\subseteq Q_{\precal}$ for $\alpha<\lambda$.

   Precaliber follows from Lemma~\ref{colorposet}(\ref{item:a}).
   So we have to show that $\lambda$ has no 1-homogeneous set (w.r.t.\ $c$) of size $\lambda$ in $V^{P_{\precal}\ast P'}$.

   Work in $V^{P_{\precal}}$ and assume that $\dot{A}$ is a $P'$-name and $p\in P'$ forces that $\dot{A}$ is in $[\lambda]^{\lambda}$. By recursion, find $A'\in[\lambda]^{\lambda}$ and $p_\zeta\leq p$ for each $\zeta\in A'$ such that $p_\zeta\Vdash \zeta\in\dot{A}$. Since $P'$ is $\lambda$-Knaster, we may assume that $\{p_\zeta : \zeta\in A'\}$ is linked. By Lemma~ \ref{forcecolor}(3), there are $\zeta\ne \zeta'$ in $A'$ such that $c(\zeta,\zeta')=0$. So there is a condition $q$ stronger that both $p_\zeta$ and $p_{\zeta'}$ forcing that $\zeta,\zeta'\in\dot{A}$ and $c(\zeta,\zeta')=0$, i.e., that $\dot A$ is not 1-homogeneous.
\end{proof}

We can now add another case to 
Lemma~\ref{lem:eleven}:
\begin{lemma}\label{pluslargemprec}
For $\aleph_1\le\lambda\leq\nu_1$ regular,
$\Ppre$
can be modified to some forcing $P'$
which still strongly witnesses the
Cicho\'n-characteristics, and additionally satisfies:
For all $k\in\omega$, $\mfrak(k\textnormal{-Knaster})=\aleph_1$;
$\mfrak(\textnormal{precaliber})=\lambda$;
and $\mathfrak p=\mathfrak b$.
\end{lemma}

\begin{proof}
The case $\lambda=\aleph_1$ was already dealt with in the previous section, so we assume $\lambda>\aleph_1$.

We modify $\Ppre$ as follows:
We start with the forcing $P_{\precal,\lambda}$.
From then on, use (by bookkeeping) all
precaliber forcings of size ${<}\lambda$,
all $\sigma$-centered ones of size ${<}\nu_3$, the value for $\bfrak$
(and in between we use all the iterands required for the original construction).
So each new iterand either
has precaliber $\aleph_1$ and is of size ${<}\lambda$, or is $(\sigma,k)$-linked for any $k\geq2$.
Therefore, the limits are $k+1$-Knaster (for any $k$).
Accordingly, the limit forces that each $k$-Knaster number is $\aleph_1$.

Also, each iterand is either of size ${<}\lambda$ or
$\sigma$-linked; so the limit is $\lambda$-Knaster, and
by Lemma~\ref{lem:blubb43} it forces that the precaliber number is
$\le\lambda$; our bookkeeping gives
$\ge\lambda$. And, as before, we get $\pfrak\ge\nu_3$ by bookkeeping.
\end{proof}

\section{Dealing with \texorpdfstring{$\mathfrak{h}$}{h}}\label{sec:h}

The following is a very useful tool to deal with $\gfrak$.

\begin{lemma}[{Blass~\cite[Thm.~2]{BlassSP}, see also Brendle~\cite[Lem.~1.17]{BrHejnice}}]\label{lemmagfrak}
   Let $\nu$ be an uncountable regular cardinal and let $( V_\alpha)_{\alpha\leq\nu}$ be an increasing sequence of transitive models of ZFC such that
   \begin{enumerate}[(i)]
      \item $\omega^\omega\cap(V_{\alpha+1}\smallsetminus V_{\alpha})\neq\emptyset$,
      \item $(\omega^\omega\cap V_\alpha)_{\alpha<\nu}\in V_\nu$, and
      \item $\omega^\omega\cap V_\nu=\bigcup_{\alpha<\nu}\omega^\omega\cap V_\alpha$.
   \end{enumerate}
   Then, in $V_\nu$, $\mathfrak g\leq\nu$.
\end{lemma}

This result gives an alternative proof of the well-known:

\begin{corollary}
   $\mathfrak{g}\leq\cof(\mathfrak{c})$.\footnote{A more elementary proof can be found in \cite[Thm.8.6, Cor. 8.7]{Blass}}
\end{corollary}
\begin{proof}
    Put $\nu:=\cof(\cfrak)$ and let $(\mu_\alpha)_{\alpha<\nu}$ be a cofinal increasing sequence in $\cfrak$ formed by limit ordinals. By recursion, we can find an increasing sequence $( V_\alpha)_{\alpha<\nu}$ of transitive models of (a large enough fragment of) ZFC such that (i) of Lemma~\ref{lemmagfrak} is satisfied, $\mu_\alpha\in V_\alpha$, $|V_\alpha|=|\mu_\alpha|$ and $\bigcup_{\alpha<\nu}\omega^\omega\cap V_\alpha=\omega^\omega$. Set $V_\nu:=V$, so Lemma~\ref{lemmagfrak} applies, i.e., $\mathfrak g\leq\nu=\cof(\cfrak)$.
\end{proof}

The following lemma is our main tool to modify the values of $\gfrak$ and $\cfrak$ via a complete subposet of some forcing, while preserving \mlike\ and Blass-uniform values from the original poset. This is a direct consequence of Lemmas~\ref{lem:blasssub} and~\ref{lemmagfrak} and
Corollary~\ref{cor:trivial}.    As we are only interested in finitely many characteristics,
the index sets $I_1$, $I_2$, $J$ and $K$  will  be finite when we apply the lemma.

\begin{lemma}\label{lem:subforcing}
Assume the following:
\begin{enumerate}[(1)]
    \item $\aleph_1\le\kappa\le\nu\le \mu$,  where $\kappa$ and $\nu$ are regular
and $\mu=\mu^{<\kappa}\geq\nu$,
    \item $P$ is a $\kappa$-cc poset forcing $\cfrak>\mu$.  
\item For some Borel relations $R^1_i$ ($i\in I_1$) on $\omega^\omega$ and some regular $\lambda^1_i\leq\mu$:
$P$ forces $\LCU_{R^1_i}(\lambda^1_i)$

\item For some Borel relations $R^2_i$ ($i\in I_2$) on $\omega^\omega$, $\lambda^2_i\leq\mu$ regular and a cardinal $\vartheta^2_i\leq\mu$: $P$ forces  $\COB_{R^2_i}(\lambda^2_i,\vartheta^2_i)$.

\item For some \mlike\ characteristics $\mathfrak y_j$ ($j\in J$)
and $\lambda_j<\kappa$:
$P\Vdash \mathfrak y_j=\lambda_j$.

\item For some \mlike\ characteristics $\mathfrak y'_k$ ($k\in K$): $P\Vdash \mathfrak{y}'_k\geq\kappa$.

\item $|I_1\cup I_2\cup J\cup K|\leq\mu$.
\end{enumerate}
Then there is
a complete subforcing $P'$ of $P$
of size
$\mu$
forcing
\begin{enumerate}[(a)]
    \item $\mathfrak y_j=\lambda_j$, $\mathfrak{y}'_k\geq\kappa$, $\LCU_{R^1_i}(\lambda^1_i)$ and $\COB_{R^2_{i'}}(\lambda^2_{i'},\vartheta^2_{i'})$ for all $i\in I_1$, $i'\in I_2$, $j\in J$ and $k\in K$;
    \item $\cfrak=\mu$ and $\gfrak\leq\nu$.
\end{enumerate}
\end{lemma}

\begin{proof}
Construct an increasing sequence of elementary submodels $(M_\alpha:\alpha<\nu)$ of some $(H(\chi),{\in}) $ for some sufficiently large $\chi$, where each $M_\alpha$ is $\mathord<\kappa$-closed with cardinality $\mu$,
in a way that $M:=M_\nu=\bigcup_{\alpha<\nu}M_\alpha$ satisfies:
\begin{enumerate}[(i)]
    \item $\mu\cup\{\mu\}\subseteq M_0$,
    \item $I_1\cup I_2\cup J\cup K\subseteq M_0$,
    \item $M_0$ contains all the definitions of the characteristics we use,
    \item $M_0$ contains all the $P$-names of witnesses of each $\LCU_{R^1_i}(\lambda^1_i)$ ($i\in I_1$),
    \item for each $i\in I_2$ and some chosen name $(\dot{\trianglelefteq}^i,\dot{\bar{g}}^i)$ of a witness of $\COB_{R^2_i}(\lambda^2_i,\vartheta^2_i)$: for all $(s,t)\in\vartheta^2_i\times\vartheta^2_i$, $\dot{g}^i_s\in M_0$ and the maximal antichain deciding ``$s\dot{\trianglelefteq}^i t$'' belongs to $M_0$,
    \item $M_{\alpha+1}$ contains $P$-names of reals that are forced not to be in the $P\cap M_\alpha$-extension (this is because $P$ forces $\cfrak>\mu$).
\end{enumerate}
Note that $M$ is also a ${<}\kappa$-closed elementary submodel of $H(\chi)$ of size $\mu$, and that $P_\alpha:=P\cap M_\alpha$ (for $\alpha\leq\nu$) is a complete subposet of $P$. Put $P':=P_\nu$.

According to Corollary~\ref{cor:trivial}, in the $P'$-extension,
each \mlike\ characteristic below $\kappa$ is preserved (as in the $P$-extension) and for the others ``$\mathfrak y'_k\ge \kappa$'' is preserved;
and according to Lemma~\ref{lem:blasssub} the
$\LCU$ and $\COB$ statements are preserved as well. This shows (a).

It is clear that $P_\alpha$ is a complete subposet of $P_\beta$ for every $\alpha<\beta\leq\nu$, and that $P'$ is the direct limit of the $P_\alpha$. Therefore, if $V'$ denotes the $P'$-extension and $V_\alpha$ denotes the $P_\alpha$-intermediate extensions, then $\omega^\omega\cap V_{\alpha+1}\smallsetminus V_\alpha\neq\emptyset$ (by (vi)) and $\omega^\omega\cap V'\subseteq\bigcup_{\alpha<\nu}V_\alpha$. Hence, by Lemma~\ref{lemmagfrak}, $V'\models\mathfrak{g}\leq\nu$. Clearly, $V'\models\cfrak=\mu$.
\end{proof}


We are now ready to add $\hfrak=\gfrak=\pfrak$ to our characteristics:

\begin{lemma}\label{lem:twelve}
For $\aleph_1\le\lambda_\mfrak\leq\kappa\le\nu_1$ regular,
$\Ppre$
can be modified to some forcing $P'$
which still strongly witnesses the
Cicho\'n-characteristics, and additionally satisfies:
   \[
     \mfrak=\lambda_\mfrak \le \hfrak=\gfrak=\pfrak=\kappa
   \]
\end{lemma}
In addition to
$\mfrak=\lambda_\mfrak$ we can get $\mfrak=\mfrak(\textnormal{precaliber})$,
which is case (3) of Lemma~\ref{lem:eleven};
and instead of $\mfrak=\lambda_\mfrak$ we can alternatively force
case (1) or (2) of Lemma~\ref{lem:eleven}, or the situation of Lemma~\ref{pluslargemprec}.
\begin{proof}
    We start with the (appropriate) 
    $P$ from Lemma~\ref{lem:eleven} (or from Lemma~\ref{pluslargemprec});
    but for the ``inflated'' continuum $\theta_\infty^+$ instead of $\theta_\infty$.

    We then apply Lemma~\ref{lem:subforcing} for $\mu:=\theta_\infty$,
    and $\nu:=\kappa$. This gives a subforcing $P'$ which still forces:
    \begin{itemize}
        \item Strong witnesses
            for all the Cicho\'n-characteristics;
            \\
            as they fall under
            Lemma~\ref{lem:subforcing}(3,4).
        \item $\pfrak\ge\kappa$; an instance of Lemma~\ref{lem:subforcing}(6) as $P$ forces $\pfrak=\nu_3\ge\kappa$.
        \item $\gfrak\le\nu$; according to Lemma~\ref{lem:subforcing}(b).
        \\
        As ZFC proves $\pfrak\le\hfrak\le\gfrak$ and $\nu=\kappa$,
        this implies
        $\pfrak=\hfrak=\gfrak=\kappa$.
        \item If $\lambda_\mfrak<\kappa$, we get
            $\mfrak=\lambda_\mfrak<\kappa$ as instance of Lemma~\ref{lem:subforcing}(5).
        \item If $\lambda_\mfrak=\kappa$, we get $\mfrak\ge\kappa$ by
            Lemma~\ref{lem:subforcing}(6);
            \\
            but as $\mfrak\le\pfrak$ this also implies $\mfrak=\lambda_\mfrak$.
        \item Alternatively: The same argument for
        $\mfrak(\textnormal{precaliber})$ and/or $\mfrak(k\textnormal{-Knaster})$ instead of / in addition to $\mfrak$;
        as required by the desired case of Lemma~\ref{lem:eleven} or~\ref{pluslargemprec}.\qedhere
    \end{itemize}
\end{proof}

We can now get twelve different characteristics:
\begin{corollary}\label{cor:twelve}
   Under Assumption~\ref{asm:main},
   and for $\aleph_1\le \lambda_\mfrak\le\kappa$ regular,
   we can get a ccc poset $P''$ which forces, in addition to Theorem~\ref{thm:old},
   \[
     \mfrak=\lambda_\mfrak \le \hfrak=\gfrak=\pfrak=\kappa
   \]
\end{corollary}
(The comment after Lemma~\ref{lem:twelve} regarding various Martins axiom numbers
applies here as well.)
\begin{proof}
    The resulting $P'$ we just constructed still satisfies the requirements
    for Lemma~\ref{lem:oldnew}, so we apply this lemma and get
    $P'':=P'\cap N^*$ (for a ${<}\kappa$-closed $N^*$)
    which forces the desired values to all Cicho\'n-characteristics.
    Additionally $P''$ forces:
    \begin{itemize}
        \item $\pfrak\ge\kappa$, by Corollary~\ref{cor:trivial}(iii)(a), as $P'$ forces $\pfrak=\kappa$.
        \item $\gfrak\le\kappa$, by Corollary~\ref{cor:trivial}(iii)(c), as $P'$ forces $\gfrak=\kappa$.
        \item $\pfrak=\hfrak=\gfrak=\kappa$, as ZFC proves $\pfrak\le\gfrak$.
        \item In case $\lambda_m<\kappa$: $\mfrak=\lambda_m$ by
        Corollary~\ref{cor:trivial}(iii)(b).
        \item In case $\lambda_m=\kappa$: $\mfrak\ge\kappa$ by
        Corollary~\ref{cor:trivial}(iii)(a), which again implies $\mfrak=\lambda_m$, as ZFC proves $\mfrak\le\pfrak=\kappa$.
    \end{itemize}
\end{proof}
\section{Products, dealing with \texorpdfstring{$\mathfrak p$}{p}}\label{sec:p}

We start reviewing a basic result in forcing theory.

\begin{lemma}[Easton's lemma]
    Let $\xi$ be an uncountable cardinal, $P$ a $\xi$-cc poset and let $Q$ be a ${<}\xi$-closed poset. Then $P$ forces that $Q$ is ${<}\xi$-distributive.
\end{lemma}
\begin{proof}
   See e.g.~\cite[Lemma 15.19]{Je03}. Note that there the lemma is
   proved for successor cardinals only, but literally the same proof
   works for any regular cardinal; for singular cardinals $\xi$
   note that ${<}\xi$-closed implies ${<}\xi^+$-closed so
   we even get ${<}\xi^+$-distributive.
\end{proof}

\begin{lemma}\label{lem:lqhwo5u25}
Assume $\xi^{<\xi}=\xi$, $P$ is $\xi$-cc,
and set $Q=\xi^{<\xi}$ (ordered by extension).
Then $P$ forces that $Q^V$ preserves all cardinals
and cofinalities.
Assume  $P\Vdash\mathfrak x=\lambda$ (in particular that
$\lambda$ is a cardinal), and let $R$ be a Borel relation.
\begin{enumerate}[(a)]
    \item If $\mathfrak x$ is \mlike:
    $\lambda<\xi$ implies $P\times Q\Vdash\mathfrak x=\lambda$;
    $\lambda\ge\xi$ implies
    $P\times Q\Vdash \mathfrak x\ge \xi$.
    \item If $\mathfrak x$ is \hlike:
    $P\times Q\Vdash\mathfrak x\le \lambda$.
    \item $P\Vdash \LCU_R(\lambda)$ implies
    $P\times Q\Vdash \LCU_R(\lambda)$.
    \item $P\Vdash \COB_R(\lambda,\mu)$ implies
    $P\times Q\Vdash \COB_R(\lambda,\mu)$.
\end{enumerate}
\end{lemma}

\begin{proof}
We call the $P^+$-extension $V''$ and the intermediate $P$-extension $V'$.

In $V'$, all $V$-cardinals $ {\ge}\xi$ are still cardinals, and
$Q$ is a $\mathord<\xi$-distributive forcing
(due to Easton's lemma).
So we can apply Lemma~\ref{lem:blassdistr} and Corollary~\ref{cor:trivial}.
%
%
%
\end{proof}

The following is shown in \cite{longlow}:

\begin{lemma}\label{fact:wreqwr}
Assume that $\xi=\xi^{<\xi}$ and $P$ is a $\xi$-cc poset that forces $\xi\le \mathfrak p$.
In the $P$-extension $V'$, let $Q=(\xi^{<\xi})^V$. Then,
\begin{enumerate}[(a)]
    \item $P\times Q=P*Q$ forces
$\mathfrak p=\xi$
    \item If in addition $P$ forces $\xi\leq\pfrak=\mathfrak{h}=\kappa$ then $P\times Q$ forces $\mathfrak{h}=\kappa$.
\end{enumerate}
\end{lemma}
\begin{proof}
    Work in the $P$-extension $V'$. $Q$ preserves cardinals and cofinalities,
    and it forces $\mathfrak{p}\ge\xi$ by Lemma~\ref{lem:lqhwo5u25}.

    There is an embedding $F$ from $\la Q,\subsetneq\ra$ into $\la[\omega]^{\aleph_0},\supsetneq^*\ra$ preserving the order and incompatibility (using the fact
    that $\xi\leq\mathfrak{p}=\mathfrak{t}$ and that every infinite
    set can be split into $\xi$ many almost disjoint sets).
    Now, $Q$
    adds a new sequence $z\in\xi^\xi\setminus V'$ and forces that $\dot{T}=\{F(z\restriction \alpha):\alpha<\xi\}$
    is a tower (hence $\mathfrak{t}\leq\xi$). If this were not the case, some condition in $Q $ would force
    that $\dot{T}$ has a pseudo-intersection $a$, but actually $a\in V'$ and it determines uniquely a branch in $\xi^\xi$, and this branch would be in fact $z$, i.e., $z\in V'$, a contradiction.
    So we have shown $P\times Q\Vdash \mathfrak t=\xi$.

    For (b): We already know that $Q\Vdash
    \mathfrak h\le \kappa$.
    To show that $\mathfrak h$ does not decrease, again work in $V'$.
    Note that $\la[\omega]^{\aleph_0},\subseteq^*\ra$ is ${<}\kappa$-closed (as $\mathfrak t=\kappa$).
    We claim that  $Q$ forces that $\la[\omega]^{\aleph_0},\subseteq^*\ra$ is ${<}\kappa$-distributive, (which implies
     $Q\Vdash \mathfrak{h}\ge \kappa$).

    If $\kappa=\xi$ then $\la[\omega]^{\aleph_0},\subseteq^*\ra$ is still ${<}\xi$-closed because $Q$ is ${<}\xi$-distributive; so assume $\xi<\kappa$. Then $Q$ is $\kappa$-cc (because $|Q|=\xi$), so $\la[\omega]^{\aleph_0},\subseteq^*\ra$ is forced to be ${<}\kappa$-distributive by Easton's Lemma (recall that $Q$ does not add new reals).
\end{proof}


We are now ready to formulate the main theorem, the consistency of
13 different values (see Figure~\ref{fig:result}):
\begin{theorem}\label{thm:main}
   Assume GCH, and that
\begin{gather*}
\aleph_1\le\lambda_\mfrak\le\xi\le\kappa
\le \lambda_\addN
\le \lambda_\covN
\le \lambda_\bfrak
\le \lambda_\nonM\le\\
\le \lambda_\covM
\le \lambda_\dfrak
\le \lambda_\nonN
\le \lambda_\cofN
\le \lambda_\infty
\end{gather*}
are regular cardinals, with the
possible exception of $\lambda_\infty$, for which we only require
$\lambda_\infty^{<\kappa}=\lambda_\infty$.
Then we can force that
\begin{gather*}
\aleph_1\le\lambda_\mfrak\leq\pfrak=\xi\le \hfrak=\gfrak=\kappa\le\\
    \addN=\lambda_\addN
\le \covN=\lambda_\covN
\le \bfrak=\lambda_\bfrak
\le \nonM=\lambda_\nonM\le\\
     \covM=\lambda_\covM
\le \dfrak=\lambda_\dfrak
\le \nonN=\lambda_\nonN
\le \cofN=\lambda_\cofN
\le 2^{\aleph_0}=\lambda_\infty
\end{gather*}
and we can additionally chose any one of the following:
\begin{itemize}
    \item $\mfrak=\mfrak(\textnormal{precaliber})=\lambda_\mfrak$.
    \item For a fixed $1\le k<\omega$, $\mfrak(k\textnormal{-Knaster})=\aleph_1$ and $\mfrak(k+1\textnormal{-Knaster})=\lambda_\mfrak$.
    \item $\mfrak(k\textnormal{-Knaster})=\aleph_1$ for all $k<\omega$,
    and $\mfrak(\textnormal{precaliber})=\lambda_\mfrak$.
\end{itemize}
\end{theorem}

\begin{proof}
Start with the appropriate forcing $P''$ of Corollary~\ref{cor:twelve}.
Then $P''\times\xi^{<\xi}$ forces:
\begin{itemize}
    \item Strong witnesses to all Cicho\'n-characteristics; by Lemma~\ref{lem:lqhwo5u25}(c,d).
    \item $\pfrak=\xi$ and $\hfrak=\kappa$; by Lemma~\ref{fact:wreqwr}.
    \item $\gfrak\le\kappa$ by
    Lemma~\ref{lem:lqhwo5u25}(b)
    as $P''$ forces $\gfrak=\kappa$ and $\gfrak$ is \hlike.
    This implies $\gfrak=\kappa$, as ZFC proves $\hfrak\le\gfrak$.
    \item The desired values to the Martin axiom numbers;
    by Lemma~\ref{lem:lqhwo5u25}(a) (and by the fact that
    $\mfrak\leq\pfrak$, in case $\lambda_m=\xi$).\qedhere
\end{itemize}
\end{proof}

\section{Alternatives}\label{sec:ext}
The methods of this paper can be used for other
initial forcings on the left hand side and for
the Boolean ultrapower method instead of the method of intersections with elementary submodels.
Also, they allow us to
compose many forcing notions with collapses
while preserving cardinal characteristics.

All these topics are described in more detail in~\cite{GKMSproc};
in the following we just give an overview.

\subsection{Another order}

In~\cite{KeShTa:1131}, another ordering of Cicho\'n's maximum is shown
to be consistent (using large cardinals), namely the ordering
shown in Figure~\ref{fig:1131}.

The initial (left hand side) forcing is based on ideas from~\cite{ShCov},
and in particular the notion of finite additive measure (FAM) limit
introduced there for random forcing.
In addition, a creature forcing $\Qhor$ similar to the one
defined in~\cite{HwSh:1067} (with Horowitz) is introduced, which forces $\nonM\ge\lambda_\nonM$ and which
has FAM-limits similar to random forcing (which is required to keep
$\mathfrak{b}$ small).

In~\cite{GKMS2}, we show that
we can remove the large cardinal assumptions for this ordering as well
(using the same method).

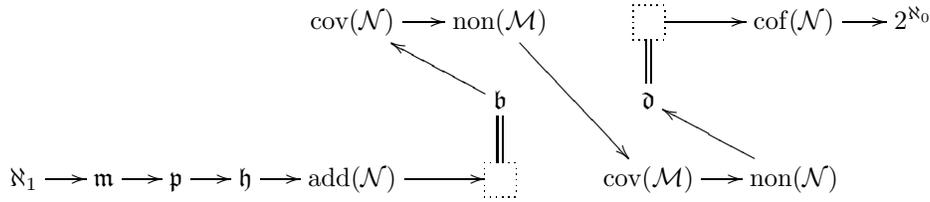
\begin{figure}
  \centering
\[
\xymatrix@=3.5ex{
&&&&            \covN\ar[r] & \nonM \ar[rdd]      &  \mye\ar@{=}[d]\ar[r]      & \cofN\ar[r] &2^{\aleph_0} \\
&&&&                               & \mathfrak b\ar[lu]  &  \mathfrak d &              \\
\aleph_1\ar[r] & \mfrak\ar[r] & \pfrak\ar[r]&\hfrak\ar[r]
& \addN\ar[r] & \mye\ar@{=}[u] &  \covM\ar[r]& \nonN\ar[lu]
}
\]
    \caption{\label{fig:1131}An alternative order that we get when we start with
    the initial forcing from~\cite{KeShTa:1131}.
    (Any $\rightarrow$ can be interpreted as either $<$ or $=$ as desired.)}
\end{figure}

It is straightforward to check that the method in this paper allows us
to add $\mfrak$, $\pfrak$, $\hfrak$ to this ordering as well;
so we get Theorem~\ref{thm:main} with both
($\bfrak$ and $\covN$) and  ($\dfrak$ and $\nonN$)
exchanged.
In particular, we get (see Figure~\ref{fig:1131}):
\begin{theorem}
   Consistently,
   \begin{gather*}
   \aleph_1<\mfrak<\pfrak<\hfrak<\addN<\bfrak<\covN<\nonM<\\<\covM<\nonN<\dfrak<\cofN<2^{\aleph_0}.
   \end{gather*}
\end{theorem}

\subsection{Boolean ultrapowers}

As mentioned in Subsection~\ref{subsec:history},
the original Cicho\'n Maximum construction \cite{GKS} uses four strongly compact cardinals: First, the left side of Cicho\'n's diagram is separated with $\Ppre$ of ~\ref{lem:oldleft}, where we assume that
there are compacts between each of $\aleph_1<\nu_1<\nu_2<\nu_3<\nu_4$.
Then four Boolean ultrapowers are applied to this poset (one for each compact cardinal) to construct a forcing $P^*$ that separates, in addition, the right hand side, while preserving the left side values already forced by $\Ppre$.

In view of Corollary~\ref{cor:trivial}(ii), we can use the methods of Sections~\ref{sec:ma}--\ref{sec:p} to force, in addition, $\mfrak<\pfrak<\hfrak<\addN$.


In contrast with Theorem~\ref{thm:main}, we can now force not only the continuum to be singular, but also $\covM$. The reason is that the poset for the left side can force $\covM=\cfrak$ singular,\footnote{This is not explicitly mentioned in~\cite{GKS}.} and the value of $\covM$ is not changed after Boolean ultrapowers (and the other methods).  The same applies to the alternate order from~\cite{KeShTa:1131} as well.

\subsection{Alternative left hand side forcings}\label{sec:alternatives}

According to subsection~\ref{subsec:history}, \cite{diegoetal} provides an alternative proof of Cicho\'n's maximum, using three strongly compact cardinals. As in~\cite{GKS}, this results from applying Boolean ultrapowers to a ccc poset that separates the left side, but the new initial forcing additionally gives  $\covM<\dfrak=\nonN=\cfrak$, where this value of $\dfrak$ can be singular. The methods of this work also apply, and we can obtain a consistency result as in Theorem~\ref{thm:main}, but there $\dfrak$ and $\cfrak$ are forced to be singular.

\subsection{Reducing gaps with collapsing forcing}

To be able to apply Boolean ultrapowers,
it is necessary to have strongly compact cardinals between the
left-hand-side values. Accordingly these values have to
have large gaps.
The methods of this paper allow to collapse these gaps; and more generally
to compose collapses with a large family of forcing notions.

For example, if $P$ forces $\mathfrak x=\lambda<\mathfrak y=\kappa$, and $\lambda$
and $\kappa$ are far apart; but you would prefer to have
$\mathfrak x=\lambda<\mathfrak y=\lambda^+$, then
the methods of~\cite{GKMSproc}
allow us to compose $P$ with a collapse of $\kappa$ to $\lambda^+$,
provided $\mathfrak x$ is reasonably well behaved.

\bibliography{morebib}
\bibliographystyle{amsalpha}
\end{document}